%% file: main.tex
\documentclass[11pt]{amsart}
\usepackage{amssymb}
\usepackage{amscd} 
\usepackage{amsfonts, amsmath, amsthm}
\usepackage{tikz-cd}
\usepackage{url}
\usepackage[a4paper, margin=1.0in]{geometry}
\usepackage{mathtools, mathrsfs, color} 
\usepackage{graphicx}
\usepackage{comment}
\usepackage{tcolorbox}
\usepackage[normalem]{ulem}
\usepackage{tikz,tikz-cd}
\usepackage[all,cmtip]{xy}
\numberwithin{equation}{section}
\usepackage{blkarray}
\usepackage{comment}
\usepackage[nobysame,alphabetic,initials,msc-links,lite,backrefs]{amsrefs}
\usepackage{hyperref}
\allowdisplaybreaks

\def\today{\number\day\space\ifcase\month\or   January\or February\or
   March\or April\or May\or June\or   July\or August\or September\or
   October\or November\or December\fi\   \number\year}

\newtheorem{thm}{Theorem}[section]
\newtheorem{lem}[thm]{Lemma}
\newtheorem{prp}[thm]{Proposition}

\newtheorem{cor}[thm]{Corollary}

\newtheorem{exa}[thm]{Example}

\newcommand{\beq}{\begin{equation}}
\newcommand{\eeq}{\end{equation}}
\newcommand{\beqr}{\begin{eqnarray*}}
\newcommand{\eeqr}{\end{eqnarray*}}
\newcommand{\bal}{\begin{align*}}
\newcommand{\eal}{\end{align*}}
\newcommand{\bei}{\begin{itemize}}
\newcommand{\eei}{\end{itemize}}

\newcommand{\Q}{{\mathbb{Q}}}
\newcommand{\Z}{{\mathbb{Z}}}
\newcommand{\R}{{\mathbb{R}}}
\newcommand{\C}{{\mathbb{C}}}

\newcommand{\T}{{\mathbb{T}}}

\newcommand{\id}{{\mathrm{id}}}

\newcommand{\Aut}{{\mathrm{Aut}}}

\newcommand{\E}{\mathcal{E}}

\newcommand{\A}{\mathscr{A}}

\title{Heisenberg-Weyl Representations and Morita equivalence for crossed products of Noncommutative solenoids}

\author[]{Pratik Kumar Kundu} 

\address{ Institute for Advancing Intelligence, TCG CREST, Sector V, Salt Lake, Kolkata 700091, India.}

\address{\href{mailto:pratik.kundu.79@tcgcrest.org}{pratik.kundu.79@tcgcrest.org}}

\keywords{Twisted group \texorpdfstring{$\rm C^*$}{C}-algebra; noncommutative solenoids; crossed product; Morita equivalence; Heisenberg--Weyl representation; \texorpdfstring{$p$}{p}-adic analysis}
\subjclass[2020]{Primary 46L55; Secondary 46L05, 46L08}

\begin{document}
\begin{abstract}
    \sloppy We study strong Morita equivalence for crossed products of noncommutative solenoids by cyclic groups. Using their description as twisted group $\mathrm C^*$-algebras, we construct automorphisms induced by the natural linear action of $\mathrm{SL}_2(\mathbb Z[1/p])$ on $\mathbb Z[1/p]^2$, for every integer $p>1$.
    For odd primes, we use the Heisenberg--Weyl representation to construct compatible cyclic group actions on the equivalence bimodules relating solenoids with reciprocal parameters. This yields strong Morita equivalences between the corresponding crossed products by $\mathbb Z$ and, for finite-order matrices, by finite cyclic groups. The corresponding equivalence for flip crossed products for the action of $\mathbb Z/2\mathbb Z$ holds for every prime, including $2$. For prime $p$, we give representatives for the nontrivial finite subgroups of $\mathrm{SL}_2(\mathbb Z[1/p])$ and explicit arithmetic criteria for conjugacy among these representatives.
\end{abstract}
%%%%%%%%%%%%%%%%%%%%%%%%%%%%%%%%%%%%%%%%%%%%%%%%%%%%%%%%%%%%%%%%%%%%%%%%%%%%%%%%%%%%%%%%%%%%%%%%%%%%%%%%%%%%%%%%%%%%%%%%%%%%%%%%%%%%%%%%%%%%%%%%%%%%%%%%%%%%%%%%%%%%%%%%%%%%%%%%%%%%%%%%%%%%%%%%%%%%%%%%%%%%%%%%%%%%%%%%%%%%%%%%%%%%%%%%%%%%%%%%%%%%%%%%%%%%%%%%%%%%%%%%%%%%%%%%%%%%%%%%%%%%%
\maketitle \pagestyle{myheadings} \markboth{Pratik~Kumar~Kundu}{Noncommutative solenoids and their crossed products}
%%%%%%%%%%%%%%%%%%%%%%%%%%%%%%%%%%%%%%%%%%%%%%%%%%%%%%%%%%%%%%%%%%%%%%%%%%%%%%%%%%%%%%%%%%%%%%%%%%%%%%%%%%%%%%%%%%%%%%%%%%%%%%%%%%%%%%%%%%%%%%%%%%%%%%%%%%%%%%%%%%%%%%%%%%%%%%%%%%%%%%%%%%%%%%%%%%%%%%%%%%%%%%%%%%%%%%%%%%%%%%%%%%%%%%%%%%%%%%%%%%%%%%%%%%%%%%%%%%%%%%%%
\input{introduction_revised.tex}

\par\medskip
\textbf{Acknowledgement:}
The author is grateful to his advisor Dr. Sayan Chakraborty for helpful discussions and valuable suggestions regarding this article. The author also thanks Dhrubajyoti Das for useful discussions on $p$-adic analysis. The author acknowledges the use of ChatGPT, developed by OpenAI, for assistance with language editing and improving the presentation of parts of this manuscript.

\input{preliminaries_revised.tex}

\input{actions_revised.tex}

\input{bimodule_revised.tex}

\input{implementers_revised.tex}

\input{equivariance_revised.tex}

\input{finite_cyclic_revised.tex}

%%%%%%%%%%%%%%%%%%%%%%%%%%%%%%%%%%%%%%%%%%%%%%%%%%%%%%%%%%%%%%%%%%%%%%%%%%%%%%%%%%%%%%%%%%%%%%%%%%%%%%%%%%%%%%%%%%%%%%%%%%%%%%%%%%%%%%%%%%%%%%%%%%%%%%%%%%%%%%%%%%%%%%%%%%%%%%%%%%%%%%%%%%%%%%%%%%%%%%%%%%%%%%%%%%%%%%%%%%%%%%%%%%%%%%%%%%%%%%%%%%%%%%%%%%%%%%%%%%%%%%%%%%%   

\end{document}

%% file: introduction_revised.tex
\section{Introduction}\label{sec:introduction}

For $\theta\in\R$, the rotation algebra $A_\theta$ is the
universal $\mathrm C^*$-algebra generated by unitaries $U,V$
satisfying $VU=e^{2\pi i\theta}UV$. It can also be realized as
a twisted group $\mathrm C^*$-algebra of $\Z^2$.
Noncommutative solenoids, introduced by Latr\'emoli\`ere and
Packer \cite{JP11}, are twisted group $\mathrm C^*$-algebras
of $\Z[1/p]^2$. For an integer $p>1$, write
\[
 R_p=\Z[1/p],\qquad \Gamma_p=R_p\times R_p,
\]
with the discrete topology. The solenoid
$\A_\alpha=C^*(\Gamma_p,\Psi_\alpha)$ is determined by a
normalized parameter $\alpha=(\alpha_n)_{n\geq0}\in\Xi_p$,
where $\Xi_p$ consists of the sequences in $[0,1)$ satisfying
\[
 p\alpha_{n+1}=\alpha_n+x_n,\qquad
 x_n\in\{0,\ldots,p-1\}.
\]
Here $\Psi_\alpha$ is the standard multiplier associated with
$\alpha$. Each solenoid is an inductive limit of rotation algebras
$A_{\alpha_{2n}}$, with connecting maps sending both canonical
generators to their $p$th powers.

The linear action of $\mathrm{SL}_2(\Z)$ on $\Z^2$ gives
rise to the actions on rotation algebras studied by Watatani
and Brenken \cites{Wat,Bre84}. Echterhoff, L\"uck, Phillips,
and Walters studied the resulting crossed products of irrational
rotation algebras by finite subgroups \cite{ELPW10}.
B\"onicke, Chakraborty, He, and Liao determined isomorphism
and Morita equivalence classes for crossed products of irrational
rotation algebras by cyclic subgroups of $\mathrm{SL}_2(\Z)$
\cite{BCHL18}. They also computed $K$-theory and tracial ranges
for crossed products associated with infinite-order matrices
and arbitrary real rotation parameters \cite{BCHL21}.
Chakraborty and Kundu studied Morita equivalence for crossed
products of higher dimensional rotation algebras by cyclic groups \cite{CK25}.
Projective modules over crossed products by finite cyclic groups
and their traces were studied by Chakraborty \cite{Cha23}.

For prime $p$, Latr\'emoli\`ere and Packer constructed
projective modules and explicit equivalence bimodules for
noncommutative solenoids \cites{JP13,JP14}. Lu developed these
constructions further and gave an explicit formula for the
reciprocal parameter of a Morita equivalent solenoid \cite{Lu22}.
The reciprocal equivalence bimodules arise from Rieffel's
Heisenberg construction \cite{Rie88}, using dual lattices
in $(\Q_p\times\R)^2$. Carrillo Rouse and Guillaume gave
a groupoid description of equivalence bimodules for
parameters related by linear fractional transformations
\cite{CRG25}. Related work on spectral triples for
noncommutative self-coverings and solenoids appears in
\cites{AGI17,FLLP24,FLP26}.

In this paper, we construct automorphisms of noncommutative
solenoids induced by the linear action of $\mathrm{SL}_2(R_p)$
on $\Gamma_p$. We then extend the reciprocal Heisenberg Morita
equivalences to crossed products by $\Z$ and by finite cyclic
groups. The standard multiplier $\Psi_\alpha$ is generally
not invariant under the linear action. Moreover, matrices
with denominators need not preserve the subgroups $p^{-n}\Z^2$
defining the rotation-algebra stages. We therefore construct
the actions directly on the solenoid using a cohomologous
invariant multiplier.

Let $\Omega_p$ consist of the real sequences
$\gamma=(\gamma_n)_{n\geq0}$ whose \emph{transition integers}
$p\gamma_{n+1}-\gamma_n$ lie in $\Z$. We write $\A_\gamma$
for the solenoid obtained by reducing its coordinates modulo $\Z$.
An \emph{even real lift} of $\gamma$ is a sequence
$\widetilde\gamma$ with even transition integers and
$\widetilde\gamma_n-\gamma_n\in\Z$ for all $n$.
Every $\gamma\in\Omega_p$ admits such a lift, which determines
an $\mathrm{SL}_2(R_p)$-invariant
alternating multiplier $\omega_{\widetilde\gamma}$ cohomologous
to $\Psi_\gamma$. In the resulting presentation
$\A_{\widetilde\gamma}^\omega
=C^*(\Gamma_p,\omega_{\widetilde\gamma})$, the action is
\[
 \lambda_A^{\widetilde\gamma}(W_z)=W_{Az}
       \qquad(A\in\mathrm{SL}_2(R_p),\ z\in\Gamma_p),
\]
where $W_z$ denotes the canonical unitary indexed by $z$.

The cohomologous multipliers give an isomorphism
$\Phi_{\widetilde\gamma}:\A_\gamma\to
\A_{\widetilde\gamma}^\omega$. It transports the action to
the standard solenoid:
\[
 \delta_A^{\widetilde\gamma}
   =\Phi_{\widetilde\gamma}^{-1}
      \lambda_A^{\widetilde\gamma}\Phi_{\widetilde\gamma}.
\]
Different even lifts give conjugate actions on $\A_\gamma$.
Consequently, for each subgroup $H\leq\mathrm{SL}_2(R_p)$,
the crossed product is independent of the lift up to
isomorphism. We also realize it as a twisted group
$\mathrm C^*$-algebra of $\Gamma_p\rtimes H$.
For $H\leq\mathrm{SL}_2(\Z)$, the action is the inductive
limit of the compatible Watatani--Brenken actions on the
rotation-algebra stages defined by the chosen even lift.

We now state the main Morita equivalence results. Let $p$
be an odd prime and let $\alpha\in\Xi_p$ satisfy
\[
 \theta=\alpha_0>0,\qquad
 x=x_\alpha:=\sum_{j\geq0}x_jp^j\in\Z_p\setminus\{0\}.
\]
Here $x_j$ are the transition integers of $\alpha$.
Write $v=v_p(x)$ and let
$x^{-1}=\sum_{j=-v}^{\infty}y_jp^j$ be its canonical
$p$-adic expansion.
Following \cite{Lu22}*{Theorem~3.6}, define the reciprocal
parameter $\beta\in\Omega_p$ by
\[
 \beta_n=\frac{\theta^{-1}
                    +\sum_{j=-v}^{n-1}y_jp^j}{p^n}
       \qquad(n\geq0).
\]
The entries of $\beta$ need not lie in $[0,1)$, which is
why we use real parameter sequences. This formula also
covers nonunits $x$, for which $v>0$.

In section~\ref{bimodule} we construct an
$\A_{\widetilde\beta}^\omega$--$\A_{\widetilde\alpha}^\omega$
imprimitivity bimodule
$\mathcal E_{\widetilde\beta,\widetilde\alpha}$ using even
lifts compatible with the Heisenberg lattice pairings.
An implementer for $A\in\mathrm{SL}_2(R_p)$ induces $A$ on the right
coefficient algebra and $B=TAT$ on the left, where
$T=\operatorname{diag}(1,-1)$. The conjugation by $T$
accounts for the sign convention in the left lattice embedding.
We write $\sim_M$ for strong Morita equivalence.

\begin{thm}\label{intro theorem Z}
 Let $p$ be an odd prime and let $\alpha\in\Xi_p$ satisfy
 $\alpha_0>0$ and $x_\alpha\ne0$. Let $\beta$ be the
 reciprocal parameter above, and choose the compatible
 even lifts $\widetilde\alpha,\widetilde\beta$ of
 Section~\ref{bimodule}. For every $A\in\mathrm{SL}_2(R_p)$,
 put $B=TAT$. The bimodule
 $\mathcal E_{\widetilde\beta,\widetilde\alpha}$ admits
 a $\Z$-action compatible with the coefficient actions
 generated by $\lambda_B^{\widetilde\beta}$ on the left
 and $\lambda_A^{\widetilde\alpha}$ on the right.
 Consequently,
 \[
 \A_{\widetilde\beta}^\omega
      \rtimes_{\lambda_B^{\widetilde\beta}}\Z
      \ \sim_M\
 \A_{\widetilde\alpha}^\omega
      \rtimes_{\lambda_A^{\widetilde\alpha}}\Z.
 \]
\end{thm}

For a finite-order matrix, the action on the bimodule can
be normalized to give an action of the corresponding
finite cyclic group.

\begin{thm}\label{intro theorem finite}
 Under the hypotheses of Theorem~\ref{intro theorem Z},
 suppose that $A$ has exact finite order $n$. Then
 $\mathcal E_{\widetilde\beta,\widetilde\alpha}$ admits
 a compatible $\Z/n\Z$-action, and
 \[
 \A_{\widetilde\beta}^\omega
      \rtimes_{\lambda_B^{\widetilde\beta}}(\Z/n\Z)
      \ \sim_M\
 \A_{\widetilde\alpha}^\omega
      \rtimes_{\lambda_A^{\widetilde\alpha}}(\Z/n\Z).
 \]
\end{thm}

We construct Weyl implementers on $L^2(\Q_p\times\R)$
from Fourier transformation, multiplication by quadratic
phases, and dilation. Weyl covariance gives compatibility
with the module operations and inner products. Powers of
one implementer define the $\Z$-action. For a matrix of
order $n$, its implementer's $n$th power is scalar, so
rescaling yields a $\Z/n\Z$-action. Equivariant Morita
theory \cites{Curto84, Combes84} then gives the
crossed-product equivalences.

We can state these equivalences directly for the standard
solenoids. For $\gamma\in\Omega_p$, write $\delta^\gamma$
for any of the conjugate actions on $\A_\gamma$ obtained by
the transport above.

\begin{cor}\label{intro corollary standard cyclic}
 Let $p$ be an odd prime and let $\alpha\in\Xi_p$ satisfy
 $\alpha_0>0$ and $x_\alpha\ne0$. Let $\beta$ be the
 reciprocal parameter above. For every $A\in\mathrm{SL}_2(R_p)$,
 put $B=TAT$. Then
 \[
 \A_\beta\rtimes_{\delta_B^\beta}\Z
     \ \sim_M\
 \A_\alpha\rtimes_{\delta_A^\alpha}\Z.
 \]
 If $A$ has exact finite order $n$, then also
 \[
 \A_\beta\rtimes_{\delta_B^\beta}(\Z/n\Z)
     \ \sim_M\
 \A_\alpha\rtimes_{\delta_A^\alpha}(\Z/n\Z).
 \]
\end{cor}

The hypotheses allow rational $\alpha_0$ and nonunit $x_\alpha$
and impose no parity condition on the original transition integers.
Since the crossed products are unital,
the induced imprimitivity bimodules are finitely generated
projective over the right crossed products.

On a standard solenoid $\A_\gamma$, the flip sends each
canonical unitary $W_z^\Psi$ to $W_{-z}^\Psi$.
For any prime $p$, if
$\alpha\in\Xi_p$ satisfies $\alpha_0>0$ and $x_\alpha\ne0$
and $\beta$ is its reciprocal parameter, then
\[
 \A_\beta\rtimes_{\mathrm{flip}_\beta}(\Z/2\Z)
      \ \sim_M\
 \A_\alpha\rtimes_{\mathrm{flip}_\alpha}(\Z/2\Z).
\]
For odd primes, this follows from
Theorem~\ref{intro theorem finite} with $A=-I$.
At $p=2$, reflection on the standard Heisenberg bimodule
gives the required equivariance; see
Theorem~\ref{thm:lu-flip-morita}.

Finally, for each prime $p$, we classify the finite subgroups
of $\mathrm{SL}_2(R_p)$. Every such subgroup is cyclic,
with possible nontrivial orders $2,3,4,6$. Every nontrivial
finite subgroup is conjugate to an integral representative
or to its conjugate by $D_p=\operatorname{diag}(p,1)$.
For order four, these two representatives are
conjugate within $\mathrm{SL}_2(R_p)$ exactly when $p=2$
or $p\equiv1\pmod4$. The order-three and order-six pairs
are conjugate exactly when $p=3$ or $p\equiv1\pmod3$.
These conjugacies give explicit isomorphisms between the
corresponding crossed products by finite cyclic groups.

Section~\ref{prelim} contains the preliminaries, and
Section~\ref{crossed product} constructs the actions and
describes their crossed products. Sections~\ref{bimodule}
and~\ref{sec:weyl-metaplectic} develop the Heisenberg
bimodule and its Weyl implementers. Section~\ref{Morita eq}
proves the result for crossed products by $\Z$.
Section~\ref{Morita eq 2} treats finite cyclic groups,
including the flip at $p=2$, and proves the subgroup
conjugacy results.

%% file: preliminaries_revised.tex
\section{Preliminaries}\label{prelim}

\subsection{Twisted group \texorpdfstring{$\mathrm C^*$}{C*}-algebras}

Let $G$ be a discrete group, written multiplicatively, with identity
$1_G$. A \emph{normalized multiplier} on $G$ is a function
$\sigma:G\times G\to\T$ satisfying
\[
 \sigma(x,y)\sigma(xy,z)
   =\sigma(x,yz)\sigma(y,z),
 \qquad
 \sigma(x,1_G)=\sigma(1_G,x)=1
\]
for all $x,y,z\in G$. Given such a multiplier, define twisted
convolution and involution on $\ell^1(G)$ by
\begin{align*}
 (f*_\sigma g)(x)=\sum_{y\in G}f(y)g(y^{-1}x)\sigma(y,y^{-1}x),\qquad
 f^{*_\sigma}(x)=\overline{\sigma(x,x^{-1})f(x^{-1})}.
\end{align*}
With these operations and the usual $\ell^1$-norm, $\ell^1(G)$ is a
Banach $*$-algebra, denoted by $\ell^1(G,\sigma)$.

A $\sigma$-representation of $G$ on a Hilbert space $\mathcal H$
is a map $V:G\to\mathcal U(\mathcal H)$ satisfying
\[
 V(x)V(y)=\sigma(x,y)V(xy)\qquad(x,y\in G).
\]
Each such representation induces a $*$-representation of
$\ell^1(G,\sigma)$ by
\[
 \pi_V(f)=\sum_{x\in G}f(x)V(x),
       \qquad f\in\ell^1(G,\sigma),
\]
where the series converges in operator norm.
The twisted left regular representation $\lambda_\sigma$ on
$\ell^2(G)$ is defined by
\[
 (\lambda_\sigma(x)\xi)(z)
   =\sigma(x,x^{-1}z)\xi(x^{-1}z)
       \qquad(x,z\in G,\ \xi\in\ell^2(G)).
\]
For $f\in\ell^1(G,\sigma)$ and $\xi\in\ell^2(G)$, its integrated
representation acts by twisted left convolution:
\[
 (\pi_{\lambda_\sigma}(f)\xi)(z)
   =\sum_{x\in G}f(x)\sigma(x,x^{-1}z)\xi(x^{-1}z)
   =(f*_\sigma\xi)(z),
\]
with $\|\pi_{\lambda_\sigma}(f)\|\leq\|f\|_1$.
The \emph{reduced twisted group $\mathrm C^*$-algebra}
$C^*(G,\sigma)$ is the completion of $\ell^1(G,\sigma)$ in the
norm $\|f\|_{\mathrm r}=\|\pi_{\lambda_\sigma}(f)\|$.
Its canonical unitaries $u_x=\lambda_\sigma(x)$ satisfy
\begin{equation}\label{eq:twisted-unitaries}
 u_{1_G}=1,\qquad
 u_xu_y=\sigma(x,y)u_{xy},\qquad
 u_x^*=\overline{\sigma(x,x^{-1})}u_{x^{-1}}.
\end{equation}

When $G$ is amenable, its full and reduced twisted group algebras
coincide, so $C^*(G,\sigma)$ has the universal presentation by unitaries
satisfying \eqref{eq:twisted-unitaries}.
This applies to $\Z^2$, $\Gamma_p$, and their semidirect products
by cyclic groups, all of which are amenable.

To compare different presentations of a solenoid, we use
cohomology of multipliers.
Two normalized multipliers $\sigma$ and $\sigma'$ are
\emph{cohomologous} if there is a function $c:G\to\T$ with $c(1_G)=1$
such that
\begin{equation}\label{coho}
 \sigma(x,y)
   =c(x)c(y)\overline{c(xy)}\,\sigma'(x,y)
   \qquad (x,y\in G).
\end{equation}
Thus changing a multiplier by a coboundary changes the phases of
the canonical unitaries, but not the twisted group algebra.

\begin{prp}[{\cite{ELPW10}*{Section~1}; \cite{BCHL21}*{Proposition~2.2}}]\label{prop:cohom}
 Let $u_x$ and $u'_x$ denote the canonical unitaries corresponding
 to $x\in G$ in $C^*(G,\sigma)$ and $C^*(G,\sigma')$, respectively.
 If $\sigma$ and $\sigma'$ satisfy \eqref{coho}, then
 \[
  u_x\longmapsto c(x)u'_x
 \]
 extends to a $*$-isomorphism
 $C^*(G,\sigma)\cong C^*(G,\sigma')$.
\end{prp}

For abelian groups, we use additive notation: the identity is $0$,
the product $xy$ becomes $x+y$, and the inverse $x^{-1}$ becomes
$-x$. All formulas above are understood in this notation when
applied to $\Z^2$ or $\Gamma_p$.

\begin{exa}\label{ex:rotation-convention}
 Recall that, for $\theta\in\R$, the two-dimensional
 \emph{noncommutative torus}, or rotation algebra, $A_\theta$ is the
 universal unital $\mathrm C^*$-algebra generated by two unitaries
 $U$ and $V$ satisfying
 \[
  VU=e(\theta)UV,
  \qquad e(t)=\exp(2\pi i t).
 \]
 We use this convention throughout.
 On $\Z^2$, define a multiplier $\omega_{\theta}:\mathbb Z^2\times\mathbb Z^2\longrightarrow \mathbb T$
 \[
  \omega_\theta((a,b),(c,d))
     =e\!\left(\frac{\theta}{2}(ad-bc)\right).
 \]
 Then $C^*(\Z^2,\omega_\theta)\cong A_\theta$ under the identification
 $u_{(1,0)}=V$ and $u_{(0,1)}=U$. More generally,
 \begin{equation}\label{eq:rotation-Weyl-convention}
  u_{(a,b)}=e(-\theta ab/2)V^aU^b.
 \end{equation}
\end{exa}

\subsection{Noncommutative solenoids and real lifts}

Let $p>1$ be an integer, and equip $R_p=\Z[1/p]$ and
$\Gamma_p=R_p^2$ with the discrete topology. Recall the parameter
space
\[
 \Xi_p=\left\{(\alpha_n)_{n\geq0}\in[0,1)^{\Z_{\geq0}}:
     p\alpha_{n+1}-\alpha_n\in\{0,\ldots,p-1\}
     \text{ for every }n\geq0\right\}.
\]
We regard $\Xi_p$ as an abelian group under coordinatewise addition
modulo $1$. The multiplier classes of $\Gamma_p$ are parametrized
by $\Xi_p$ as follows.

\begin{thm} {\cite{JP11}*{Theorem~2.3}}
\label{thm:solenoid-multipliers}
 Every multiplier on $\Gamma_p$ is cohomologous to
 $\Psi_\alpha$ for a unique $\alpha\in\Xi_p$, where
 \begin{equation}\label{eq:standard-solenoid-multiplier}
 \Psi_\alpha\left(
   \left(\frac{j_1}{p^{k_1}},\frac{j_2}{p^{k_2}}\right),
   \left(\frac{j_3}{p^{k_3}},\frac{j_4}{p^{k_4}}\right)
 \right)
   =e(\alpha_{k_1+k_4}j_1j_4),
 \end{equation}
 for $j_i\in\Z$ and $k_i\geq0$.
\end{thm}

For $\alpha\in\Xi_p$, the twisted group algebra
\[
 \A_\alpha=C^*(\Gamma_p,\Psi_\alpha)
\]
is called a \emph{noncommutative solenoid}. The uniqueness in
Theorem~\ref{thm:solenoid-multipliers} concerns multiplier classes;
distinct parameters may define isomorphic algebras.

The transition relation $p\alpha_{n+1}-\alpha_n\in\Z$ ensures that
\eqref{eq:standard-solenoid-multiplier} is independent of the chosen
representations of the fractions. The resulting function is a
bicharacter and hence a normalized multiplier.

For the constructions below, we also use the group of real lifts
\[
 \Omega_p=\left\{(\gamma_n)_{n\geq0}\in\R^{\Z_{\geq0}}:
     p\gamma_{n+1}-\gamma_n\in\Z
     \text{ for every }n\geq0\right\},
\]
with coordinatewise addition. For $\gamma\in\Omega_p$, we call
the integers $p\gamma_{n+1}-\gamma_n$ its
\emph{transition integers}. For $\gamma\in\Xi_p$, these belong
to $\{0,\ldots,p-1\}$.
Reduction modulo $\Z$ defines a
surjective homomorphism
\[
 h:\Omega_p\longrightarrow\Xi_p,\qquad
 h(\gamma)_n=\gamma_n-\lfloor\gamma_n\rfloor.
\]
Indeed, for $\gamma\in\Omega_p$ and $n\geq0$,
\[
 p\,h(\gamma)_{n+1}-h(\gamma)_n
   \in\Z\cap(-1,p)=\{0,\ldots,p-1\},
\]
so $h(\gamma)\in\Xi_p$. Since $h|_{\Xi_p}=\id_{\Xi_p}$,
the map $h$ is surjective. Formula
\eqref{eq:standard-solenoid-multiplier} extends to $\Omega_p$ and
satisfies $\Psi_\gamma=\Psi_{h(\gamma)}$. Accordingly, we set
\[
 \A_\gamma=C^*(\Gamma_p,\Psi_\gamma)=\A_{h(\gamma)}
       \qquad(\gamma\in\Omega_p).
\]
Although the standard multiplier depends only on $h(\gamma)$,
the alternating presentation introduced in
Section~\ref{crossed product} uses a choice of real lift.

The following inductive-limit description is
\cite{JP11}*{Theorem~3.6}, expressed in our generator convention and
extended to real lifts.

\begin{thm}\label{thm:solenoid-inductive-limit}
 Let $\alpha\in\Omega_p$. For each $n\geq0$, let
 $U_{2n},V_{2n}$ be the canonical generators of $A_{\alpha_{2n}}$,
 satisfying
 $V_{2n}U_{2n}=e(\alpha_{2n})U_{2n}V_{2n}$. The assignments
 \begin{equation}\label{connecting mor}
  \varphi_n(U_{2n})=U_{2n+2}^{p},\qquad
  \varphi_n(V_{2n})=V_{2n+2}^{p}
 \end{equation}
 define injective $*$-homomorphisms, and
 \[
  \A_\alpha\cong
    \varinjlim(A_{\alpha_{2n}},\varphi_n).
 \]
 Writing $W_z^\Psi$ for the canonical unitaries of $\A_\alpha$,
 the canonical embeddings $\nu_n:A_{\alpha_{2n}}\to\A_\alpha$ are
 \[
  \nu_n(V_{2n})=W_{(p^{-n},0)}^\Psi,\qquad
  \nu_n(U_{2n})=W_{(0,p^{-n})}^\Psi.
 \]
\end{thm}

\begin{proof}
 Since $p^2\alpha_{2n+2}-\alpha_{2n}\in\Z$, the $p$th powers of
 $U_{2n+2}$ and $V_{2n+2}$ satisfy the defining relation of
 $A_{\alpha_{2n}}$. Thus $\varphi_n$ is well defined by the universal
 property. Likewise, \eqref{eq:standard-solenoid-multiplier} gives
 \begin{equation}\label{eq:stage-relation}
 W_{(p^{-n},0)}^\Psi W_{(0,p^{-n})}^\Psi
  =e(\alpha_{2n})
     W_{(0,p^{-n})}^\Psi W_{(p^{-n},0)}^\Psi,
 \end{equation}
 so the displayed formulas define $\nu_n$.

 Let $\tau_\alpha$ and $\tau_{\alpha_{2n}}$ denote the canonical
 traces on $\A_\alpha$ and $A_{\alpha_{2n}}$, respectively. Since
 \[
  \tau_\alpha\circ\nu_n=\tau_{\alpha_{2n}}
 \]
 and $\tau_{\alpha_{2n}}$ is faithful, $\nu_n$ is injective.

 As $\Psi_\alpha$ is trivial on each coordinate axis,
 $\nu_{n+1}\circ\varphi_n=\nu_n$. Hence each $\varphi_n$ is
 injective and the stage images form an increasing sequence.
 Finally, if $z=(a/p^n,b/p^n)$ with $a,b\in\Z$, then
 \[
  W_z^\Psi=e(-\alpha_{2n}ab)\,
       \nu_n(V_{2n}^{a}U_{2n}^{b}).
 \]
 Since $\Gamma_p=\bigcup_{n\geq0}p^{-n}\Z^2$, every canonical
 unitary of $\A_\alpha$ therefore belongs to
 $\bigcup_{n\geq0}\nu_n(A_{\alpha_{2n}})$. These unitaries generate
 $\A_\alpha$, so the increasing union of the stage images is dense
 in $\A_\alpha$. The compatible injective maps $\nu_n$ thus induce
 the asserted isomorphism with the inductive limit.
\end{proof}

\subsection{\texorpdfstring{$p$}{p}-adic preliminaries}

Fix a prime $p$. For $q\in\Q_p\setminus\{0\}$, write its canonical
expansion as
\[
 q=\sum_{j=v}^{\infty}q_jp^j,\qquad
 q_j\in\{0,\ldots,p-1\},\quad q_v\neq0.
\]
Set $v_p(q)=v$ and $|q|_p=p^{-v}$, with $v_p(0)=\infty$ and
$|0|_p=0$. The ring of $p$-adic integers is
$\Z_p=\{q\in\Q_p:|q|_p\leq1\}$. We use the natural embeddings of
$\Z[1/p]$ into $\R$ and $\Q_p$. In particular,
\begin{equation}\label{eq:Rp-intersection}
 \Z[1/p]\cap\Z_p=\Z.
\end{equation}
Indeed, for $m\in\Z$ and $k\geq0$, the condition $m/p^k\in\Z_p$
is equivalent to $p^k\mid m$. We refer to \cite{Robert2000} for
background on $p$-adic analysis.

\begin{lem}\label{inverse-multiply-one}
 Let $0\neq x=\sum_{j=v}^{\infty}x_jp^j\in\Z_p$, where
 $v=v_p(x)\geq0$, and write
 $x^{-1}=\sum_{j=-v}^{\infty}y_jp^j$, with the canonical digits in
 both expansions. For every $k\geq0$,
 \[
  \left(\sum_{j=v}^{v+k}x_jp^{j-v}\right)
  \left(\sum_{j=-v}^{-v+k}y_jp^{j+v}\right)
   \equiv1\pmod{p^{k+1}}.
 \]
 These congruences determine the digits $y_j$ successively.
\end{lem}

\begin{proof}
 Set $u=p^{-v}x\in\Z_p^\times$ and $w=p^vx^{-1}=u^{-1}$.
 Reducing $uw=1$ modulo $p^{k+1}$ gives the displayed congruence.
 Since $u$ is a unit, the residue of $w$ modulo $p^{k+1}$ is
 uniquely determined. These residues are compatible as $k$
 increases and hence determine the digits of $w$, and therefore
 those of $x^{-1}$, successively.
\end{proof}

For nonzero $q=\sum_{j=v}^{\infty}q_jp^j$, define the $p$-adic
fractional part by
\[
 \{q\}_p=\sum_{j=v}^{-1}q_jp^j\in \Z[1/p]\cap[0,1),
\]
where the sum is zero for $v\geq0$, and put $\{0\}_p=0$. Since
$q-\{q\}_p\in\Z_p$, we have
\[
 \{q+r\}_p-\{q\}_p-\{r\}_p\in \Z[1/p]\cap\Z_p=\Z
       \qquad(q,r\in\Q_p).
\]
Consequently,
\[
 e_p:\Q_p\longrightarrow\T,\qquad e_p(q)=e(\{q\}_p),
\]
is a continuous additive character with kernel $\Z_p$. Our sign
convention gives
\begin{equation}\label{eq:ep-on-Rp}
 e_p(r)=e(r)\qquad(r\in \Z[1/p]).
\end{equation}
Indeed, $r-\{r\}_p\in \Z[1/p]\cap\Z_p=\Z$.

\begin{lem}\label{fractional-part-congruence}
 Let $x=\sum_{j=v}^{\infty}x_jp^j\in\Q_p\setminus\{0\}$, and
 let $s_i=m_i/p^{k_i}\in \Z[1/p]$ for $i=1,2$, where
 $k_i\geq0$. Set $K=k_1+k_2$ and
 $X_K=\sum_{j=v}^{K-1}x_jp^j$, with $X_K=0$ if
 $v\geq K$. Then
 \[
  \{xs_1s_2\}_p\equiv X_Ks_1s_2\pmod\Z,
  \qquad
  e_p(xs_1s_2)=e(X_Ks_1s_2).
 \]
\end{lem}

\begin{proof}
 Since $x-X_K\in p^K\Z_p$, we have
 $(x-X_K)s_1s_2\in\Z_p$. It follows that
 \[
  \{xs_1s_2\}_p-X_Ks_1s_2\in \Z[1/p]\cap\Z_p=\Z,
 \]
 proving the congruence. Applying $e$ gives the character identity.
\end{proof}

Let $dq$ be Haar measure on $\Q_p$, normalized by $dq(\Z_p)=1$.
This measure is self-dual for the pairing
$(q,q')\mapsto e_p(qq')$. Let $\mathcal S(\Q_p)$ denote the space
of locally constant, compactly supported functions on $\Q_p$, and
let $\mathcal S(\R)$ be the usual Schwartz space. For
$M=\Q_p\times\R$, we use the Schwartz--Bruhat space
\[
 \mathcal S(M)=\mathcal S(\Q_p)\otimes\mathcal S(\R),
\]
where the tensor product is algebraic. We equip $\mathcal{S}(M)$ with its standard locally convex inductive-limit topology. This space is dense in
$L^2(M,dq\,dt)$ and is invariant under translations, multiplication
by characters, and Fourier transformation; see \cite{Dei12}.
The pairing and self-dual measure for the Heisenberg bimodule
are specified in Section~\ref{bimodule}.

%% file: actions_revised.tex
\section{The matrix action on noncommutative solenoids and crossed products}
\label{crossed product}

Throughout this section, $p>1$ is an integer and all groups are
discrete. Let $\mathrm{SL}_2(R_p)$ denote the group of
$2\times2$ matrices with entries in $R_p=\Z[1/p]$ and
determinant $1$. It acts on $\Gamma_p=R_p^2$ by multiplication
on column vectors. We construct an action of this group on every
noncommutative solenoid using an invariant alternating multiplier,
and identify the resulting crossed products with twisted group
algebras of semidirect products.

One can observe that the standard multiplier $\Psi_\alpha$ is
not invariant under this matrix action in general. For example, if
$L=\begin{psmallmatrix}1&0\\1&1\end{psmallmatrix}$ and
$z=(1,0)$, then $\Psi_\alpha(z,z)=1$, whereas
$\Psi_\alpha(Lz,Lz)=e(\alpha_0)$. The matrix $-I$, however,
preserves $\Psi_\alpha$ and defines the standard flip.

\begin{prp}\label{prop:standard-flip-action}
 For every $\alpha\in\Omega_p$, the formula
 \[
   \mathrm{flip}_\alpha(W_z^\Psi)=W_{-z}^\Psi
       \qquad(z\in\Gamma_p)
 \]
 defines an involutive automorphism of $\A_\alpha$.
 Thus it defines an action of $\Z/2\Z$ on $\A_\alpha$.
\end{prp}
\begin{proof}
 The bicharacter identity in \eqref{eq:standard-solenoid-multiplier} gives
 $\Psi_\alpha(-z,-w)=\Psi_\alpha(z,w)$ for all
 $z,w\in\Gamma_p$. Hence $W_z^\Psi\mapsto W_{-z}^\Psi$
 preserves the twisted multiplication and adjoints. Applying it
 twice fixes every canonical unitary, so it is an involutive
 automorphism.
\end{proof}

\subsection{Even real lifts and an invariant multiplier}
\label{action}

Recall the map $h:\Omega_p\to\Xi_p$ from Section~\ref{prelim}, and put
\[
 \Omega_p^{\mathrm{even}}
  =\{\alpha\in\Omega_p:
      p\alpha_{n+1}-\alpha_n\in2\Z\text{ for every }n\geq0\}.
\]
For $\gamma\in\Xi_p$, an \emph{even real lift} is an element
$\alpha\in\Omega_p^{\mathrm{even}}$ satisfying $h(\alpha)=\gamma$.
The following lemma establishes the existence of such lifts.

\begin{lem}\label{lem:even-lift-exists}
 The map $h$ sends $\Omega_p^{\mathrm{even}}$ onto $\Xi_p$.
\end{lem}
\begin{proof}
 Let $\alpha\in\Xi_p$, and write
 $x_n=p\alpha_{n+1}-\alpha_n\in\Z$. For integers $m_n$, set
 $\widetilde\alpha_n=\alpha_n+m_n$. Then
 \[
 p\widetilde\alpha_{n+1}-\widetilde\alpha_n
   =x_n+pm_{n+1}-m_n.
 \]
 If $p$ is odd, take $m_0=0$ and recursively choose
 $m_{n+1}\in\{0,1\}$ with
 $m_{n+1}\equiv m_n-x_n\pmod2$. If $p$ is even, choose
 $m_n\in\{0,1\}$ independently with
 $m_n\equiv x_n\pmod2$. In either case,
 $p\widetilde\alpha_{n+1}-\widetilde\alpha_n\in2\Z$ and
 $h(\widetilde\alpha)=\alpha$.
\end{proof}

\begin{exa}\label{ex:even-lift-family}
 For $\vartheta\in\R$, the sequence
 $\alpha_n=\vartheta/p^n$ belongs to $\Omega_p^{\mathrm{even}}$.
 More generally, choosing arbitrary even integers $x_j$ and
 setting
 \[
 \alpha_n=\frac{\vartheta+\sum_{j=0}^{n-1}x_jp^j}{p^n}
 \]
 gives $p\alpha_{n+1}-\alpha_n=x_n$. If
 $0<\vartheta<1$ and $0\leq x_j<p$, then $\alpha=(\alpha_n)_{n\geq0}\in\Xi_p$.
\end{exa}

An even real lift determines a character of $R_p$, from which we
obtain an invariant multiplier on $\Gamma_p$.

\begin{prp}\label{prop:invariant-multiplier}
 Let $\alpha\in\Omega_p^{\mathrm{even}}$. Then the map $h_\alpha:R_p\to\T$ defined by
 \[
 h_\alpha\!\left(\frac{j}{p^k}\right)
   =e\!\left(\frac{\alpha_kj}{2}\right),
   \qquad j\in\Z,\ k\geq0,
 \]
 is a character. For
 $z=(a,b)$ and $w=(c,d)$ in $\Gamma_p$, put
 \begin{equation}\label{eq:alternating-solenoid-multiplier}
 \omega_\alpha(z,w)=h_\alpha(\det(z,w))
                    =h_\alpha(ad-bc).
 \end{equation}
 Then $\omega_\alpha$ is a normalized alternating
 bicharacter, is invariant under $\mathrm{SL}_2(R_p)$,
 and is cohomologous to $\Psi_\alpha$. In particular,
 $C^*(\Gamma_p,\omega_\alpha)\cong\A_\alpha$.
\end{prp}
\begin{proof}
 Replacing $j/p^k$ by $pj/p^{k+1}$ changes the exponent
 in $h_\alpha$ by
 $j(p\alpha_{k+1}-\alpha_k)/2\in\Z$, so the value is unchanged.
 Iterating this replacement allows any two representations of the
 same element to be compared at a common denominator. Hence
 $h_\alpha$ is well defined. Given $r,s\in R_p$, write
 $r=j/p^N$ and $s=\ell/p^N$. Then
 \[
 h_\alpha(r+s)
   =e\!\left(\frac{\alpha_N(j+\ell)}2\right)
   =h_\alpha(r)h_\alpha(s).
 \]
 Since $h_\alpha(0)=1$, this proves that $h_\alpha$ is a character.

 The determinant is additive in each variable, so
 $\omega_\alpha$ is a bicharacter. In particular, for
 $z,w,v\in\Gamma_p$, we have
 \[
 \begin{aligned}
 \omega_\alpha(z,w)\omega_\alpha(z+w,v)
   &=\omega_\alpha(z,w)\omega_\alpha(z,v)\omega_\alpha(w,v)=\omega_\alpha(z,w+v)\omega_\alpha(w,v),
 \end{aligned}
 \]
 proving the cocycle identity. It is normalized because
 $\det(z,0)=\det(0,z)=0$, and alternating because
 $\det(z,z)=0$. For $A\in\mathrm{SL}_2(R_p)$,
 $\det(Az,Aw)=\det(A)\det(z,w)=\det(z,w)$, whence
 \begin{equation}\label{eq:multiplier-invariance}
 \omega_\alpha(Az,Aw)=\omega_\alpha(z,w)
             \qquad(A\in\mathrm{SL}_2(R_p)).
 \end{equation} 
 Hence the multiplier $\omega_{\alpha}$ is invariant under $\mathrm{SL}_2(\Z[1/p])$ action.
 
 To compare $\omega_\alpha$ with $\Psi_\alpha$, define
 $c_\alpha(a,b)=h_\alpha(-ab)$, so that $c_\alpha(0)=1$.
 For $z=(a,b)$ and $w=(c,d)$,
 \[
 \begin{aligned}
  c_\alpha(z)c_\alpha(w)\overline{c_\alpha(z+w)}
  &=h_\alpha\bigl(-ab-cd+(a+c)(b+d)\bigr)
  &=h_\alpha(ad+bc).
 \end{aligned}
 \]
 Moreover, writing $a=j_1/p^{k_1}$ and $d=j_4/p^{k_4}$ gives
 \[
 h_\alpha(2ad)
   =e(\alpha_{k_1+k_4}j_1j_4)
   =\Psi_\alpha(z,w)
 \]
 by \eqref{eq:standard-solenoid-multiplier}. Consequently,
 \begin{equation}\label{eq:standard-alternating-cohomology}
 c_\alpha(z)c_\alpha(w)\overline{c_\alpha(z+w)}\omega_\alpha(z,w)
   =h_\alpha(2ad)=\Psi_\alpha(z,w).
 \end{equation}
 The claimed $*$-isomorphism follows from
 Proposition~\ref{prop:cohom}.
\end{proof}

For
 $z=(j_1/p^{k_1},j_2/p^{k_2})$ and $w=(j_3/p^{k_3},j_4/p^{k_4})$,
the multiplier $\omega_\alpha$ has the explicit form
\[
 \omega_\alpha(z,w)
  =e\!\left(
    \frac{\alpha_{k_1+k_4}j_1j_4
          -\alpha_{k_2+k_3}j_2j_3}{2}
      \right).
\]
Write $\A_\alpha^\omega=C^*(\Gamma_p,\omega_\alpha)$
and denote its canonical unitary at $z$ by $W_z$.
When comparing alternating presentations, we write $W_z^\alpha$
to indicate the parameter explicitly.
Alternation gives $W_z^*=W_{-z}$. If $W_z^\Psi$
denotes the canonical unitary in the standard model, the
isomorphism supplied by \eqref{eq:standard-alternating-cohomology}
is
\begin{equation}\label{eq:Phi-standard-to-alternating}
 \Phi_\alpha:\A_\alpha\longrightarrow\A_\alpha^\omega,
 \qquad \Phi_\alpha(W_z^\Psi)=c_\alpha(z)W_z.
\end{equation}
The standard multiplier depends only on $h(\alpha)$, whereas
$\omega_\alpha$ is defined using the chosen even lift. For each such
lift, \eqref{eq:multiplier-invariance} shows that the natural action
of $\mathrm{SL}_2(R_p)$ on $\Gamma_p$ preserves the defining
relations of $\A_\alpha^\omega$ and hence induces an action on this
algebra, as follows.

\begin{prp}\label{prop:matrix-action}
 Let $\alpha\in\Omega_p^{\mathrm{even}}$. For
 $A\in\mathrm{SL}_2(R_p)$, the assignment
 \begin{equation}\label{eq:matrix-action}
 \lambda_A^\alpha(W_z)=W_{Az},\qquad z\in\Gamma_p,
 \end{equation}
 extends to an automorphism of $\A_\alpha^\omega$.
 Moreover, $\lambda_A^\alpha\lambda_B^\alpha
  =\lambda_{AB}^\alpha$.
\end{prp}
\begin{proof}
 Fix $A\in\mathrm{SL}_2(R_p)$. Define a unitary $U_A$ on
 $\ell^2(\Gamma_p)$ by $(U_A\xi)(z)=\xi(A^{-1}z)$.
 Identifying $W_z$ with $\lambda_{\omega_\alpha}(z)$,
 multiplier invariance gives
 \[
 U_AW_zU_A^*=W_{Az}\qquad(z\in\Gamma_p).
 \]
 Conjugation by $U_A$ therefore defines a unital
 $*$-homomorphism
 $\lambda_A^\alpha:\A_\alpha^\omega\to\A_\alpha^\omega$
 with $\lambda_A^\alpha(W_z)=W_{Az}$.

 For $A,B\in\mathrm{SL}_2(R_p)$ and $z\in\Gamma_p$,
 \[
 (\lambda_A^\alpha\circ\lambda_B^\alpha)(W_z)
   =\lambda_A^\alpha(W_{Bz})
   =W_{ABz}
   =\lambda_{AB}^\alpha(W_z).
 \]
 Since the canonical unitaries generate $\A_\alpha^\omega$,
 this proves $\lambda_A^\alpha\circ\lambda_B^\alpha
 =\lambda_{AB}^\alpha$ on the whole algebra.

 Finally, $\lambda_I^\alpha=\id$, and $A^{-1}$ belongs to
 $\mathrm{SL}_2(R_p)$. The group law therefore gives
 $\lambda_A^\alpha\circ\lambda_{A^{-1}}^\alpha
 =\lambda_{A^{-1}}^\alpha\circ\lambda_A^\alpha=\id$.
 Hence each $\lambda_A^\alpha$ is an automorphism, with inverse
 $\lambda_{A^{-1}}^\alpha$, and
 $A\mapsto\lambda_A^\alpha$ defines the required action.
\end{proof}

The construction of $\omega_\alpha$ uses a chosen even lift of
$h(\alpha)$, and different lifts may yield different multipliers.
We now show that, for lifts of the same parameter, the resulting
algebras admit a $*$-isomorphism intertwining their matrix actions.

\begin{prp}\label{prop:lift-independence}
 Suppose $\alpha,\alpha'\in\Omega_p^{\mathrm{even}}$
 have the same reduction in $\Xi_p$. Then there is a $*$-isomorphism
 \[
 F:\A_{\alpha'}^\omega\longrightarrow\A_\alpha^\omega
 \]
 that is equivariant for the matrix actions, in the sense that
 \[
 F\circ\lambda_A^{\alpha'}=\lambda_A^\alpha\circ F
       \qquad(A\in\mathrm{SL}_2(R_p)).
 \]
\end{prp}
\begin{proof}
 Since $h(\alpha')=h(\alpha)$, the differences
 $k_n=\alpha'_n-\alpha_n$ are integers. The even-lift conditions
 give
 \[
 pk_{n+1}-k_n
   =(p\alpha'_{n+1}-\alpha'_n)
     -(p\alpha_{n+1}-\alpha_n)\in2\Z.
 \]
 If $p$ is even, this implies $k_n\in2\Z$ for every $n$.
 The definition of the characters then gives
 $h_{\alpha'}=h_\alpha$, so $\omega_{\alpha'}=\omega_\alpha$.
 The canonical identification $F(W_z^{\alpha'})=W_z^\alpha$
 therefore intertwines the matrix actions.

 Suppose now that $p$ is odd. The same relation gives
 $k_{n+1}\equiv k_n\pmod2$, so all $k_n$ have a common parity
 $\varepsilon\in\{0,1\}$. Reduction modulo $2$ on $\Z$ extends
 to a ring homomorphism $R_p\to\mathbb F_2$, since the image
 of $p$ is invertible. Explicitly,
 \[
 \overline{j/p^n}=(j\bmod2)(p\bmod2)^{-n}=j\bmod2.
 \]
 For $r=j/p^n$,
 \[
 \frac{h_{\alpha'}(r)}{h_\alpha(r)}
      =e(k_nj/2)=(-1)^{\varepsilon\bar r}.
 \]
 Applying this identity to $\det(z,w)$, for $z=(a,b)$ and
 $w=(c,d)$, gives
 \begin{equation}\label{eq:lift-multiplier-ratio}
 \frac{\omega_{\alpha'}(z,w)}{\omega_\alpha(z,w)}
   =(-1)^{\varepsilon(\bar a\bar d+\bar b\bar c)}.
 \end{equation}
 Define $q:\mathbb F_2^2\to\mathbb F_2$ by $q(u,v)=uv+u+v$,
 and put
 \[
 \eta(z)=(-1)^{\varepsilon q(\bar z)},
 \qquad \bar z=(\bar a,\bar b)\quad\text{for }z=(a,b).
 \]
 Then $\eta(0)=1$. Expanding $q$ over $\mathbb F_2$ gives
 \[
 \begin{aligned}
 (\delta\eta)(z,w)
   &=\eta(z)\eta(w)\overline{\eta(z+w)}=(-1)^{\varepsilon(q(\bar z)+q(\bar w)-q(\bar z+\bar w))}=(-1)^{\varepsilon(\bar a\bar d+\bar b\bar c)}.
 \end{aligned}
 \]
 By \eqref{eq:lift-multiplier-ratio},
 $\omega_{\alpha'}=(\delta\eta)\omega_\alpha$.
 Proposition~\ref{prop:cohom} therefore gives a $*$-isomorphism
 $F$ satisfying
 \[
 F(W_z^{\alpha'})=\eta(z)W_z^\alpha.
 \]

 To verify equivariance, observe that $q$ is zero at $(0,0)$
 and one at each nonzero vector of $\mathbb F_2^2$.
 For $A\in\mathrm{SL}_2(R_p)$, its entrywise reduction
 $\bar A\in\mathrm{SL}_2(\mathbb F_2)$ fixes zero and permutes
 the nonzero vectors. Hence $q(\overline{Az})=q(\bar z)$ and
 $\eta(Az)=\eta(z)$. Consequently,
 \[
 \begin{aligned}
 (F\circ\lambda_A^{\alpha'})(W_z^{\alpha'})=\eta(Az)W_{Az}^\alpha=\eta(z)W_{Az}^\alpha=(\lambda_A^\alpha\circ F)(W_z^{\alpha'}).
 \end{aligned}
 \]
 Since the canonical unitaries generate $\A_{\alpha'}^\omega$,
 this proves $F\circ\lambda_A^{\alpha'}=\lambda_A^\alpha\circ F$
 on the whole algebra.
\end{proof}

\subsection{Crossed products of noncommutative solenoids}

We next transport the matrix action to the standard solenoid and
compare the resulting crossed products.

Let $D$ be a unital $\mathrm C^*$-algebra and
$\eta:H\to\Aut(D)$ an action of a discrete group. Choose a faithful
unital representation $\pi:D\to\mathcal B(\mathcal K)$. The regular
covariant representation on $\ell^2(H,\mathcal K)$ is given by
\[
 (\widetilde\pi(d)\xi)(h)
   =\pi(\eta_{h^{-1}}(d))\xi(h),\qquad
 (v_k\xi)(h)=\xi(k^{-1}h).
\]
The reduced crossed product $D\rtimes_\eta H$ is the norm closure
of the linear span of $\widetilde\pi(d)v_h$, for $d\in D$ and
$h\in H$. It is independent of the choice of faithful $\pi$.
We identify $D$ with its image under $\widetilde\pi$; the canonical
group unitaries satisfy
\[
 v_h d v_h^*=\eta_h(d)\qquad(d\in D,\ h\in H).
\]
If $H$ is amenable, this algebra coincides with the full crossed
product and therefore has the usual universal property.
We use the unadorned notation $\rtimes$ for the reduced completion
throughout; see \cite{Wil07}.

Let $\gamma\in\Omega_p$, and choose
$\widetilde\gamma\in\Omega_p^{\mathrm{even}}$ with
$h(\widetilde\gamma)=h(\gamma)$. Since
$\Psi_{\widetilde\gamma}=\Psi_\gamma$, the map
$\Phi_{\widetilde\gamma}$ in
\eqref{eq:Phi-standard-to-alternating} identifies
$\A_\gamma$ with $\A_{\widetilde\gamma}^\omega$.
Conjugating the matrix action by this isomorphism defines an action
on the standard model:
\[
 \delta_C^{\widetilde\gamma}
  =\Phi_{\widetilde\gamma}^{-1}\circ
    \lambda_C^{\widetilde\gamma}\circ\Phi_{\widetilde\gamma},
       \qquad C\in\mathrm{SL}_2(R_p).
\]
Indeed, $\delta_I^{\widetilde\gamma}=\id$ and
$\delta_C^{\widetilde\gamma}\delta_D^{\widetilde\gamma}
 =\delta_{CD}^{\widetilde\gamma}$. On the canonical unitaries,
\begin{equation}\label{eq:transported-matrix-action}
 \delta_C^{\widetilde\gamma}(W_z^\Psi)
   =c_{\widetilde\gamma}(z)
      \overline{c_{\widetilde\gamma}(Cz)}W_{Cz}^\Psi,
       \qquad C\in\mathrm{SL}_2(R_p).
\end{equation}
In particular, if
$C=\begin{psmallmatrix}a&b\\c&d\end{psmallmatrix}$ and
$z=(r,s)$, the scalar in \eqref{eq:transported-matrix-action} is
\[
 h_{\widetilde\gamma}\bigl((ar+bs)(cr+ds)-rs\bigr).
\]
By construction, $\Phi_{\widetilde\gamma}$ intertwines the two actions:
\[
 \Phi_{\widetilde\gamma}\circ\delta_C^{\widetilde\gamma}
   =\lambda_C^{\widetilde\gamma}\circ\Phi_{\widetilde\gamma}.
\]
The following corollary shows that different even lifts give
conjugate actions on the standard model and isomorphic crossed
products.

\begin{cor}\label{cor:lift-independent-crossed-products}
 Let $\gamma\in\Omega_p$, and let
 $\widetilde\gamma,\widehat\gamma\in
 \Omega_p^{\mathrm{even}}$ have the same reduction
 $h(\gamma)$. Then
 $\Psi_{\widetilde\gamma}
   =\Psi_{\widehat\gamma}=\Psi_\gamma$, so both lifts
 give the same standard solenoid $\A_\gamma$.
 The actions $\delta^{\widetilde\gamma}$ and
 $\delta^{\widehat\gamma}$ on this algebra are conjugate
 as $\mathrm{SL}_2(R_p)$-actions. Consequently, for every
 subgroup $H\leq\mathrm{SL}_2(R_p)$,
 \[
 \A_\gamma\rtimes_{\delta^{\widehat\gamma}}H
   \ \cong\
 \A_\gamma\rtimes_{\delta^{\widetilde\gamma}}H
   \ \cong\
 \A_{\widetilde\gamma}^\omega
       \rtimes_{\lambda^{\widetilde\gamma}}H.
 \]
 In particular, for fixed $h(\gamma)$ and $H$, the crossed product
 is independent of the even lift up to $*$-isomorphism, although the action
 formulas on the fixed standard generators need not coincide.
\end{cor}
\begin{proof}
 Proposition~\ref{prop:lift-independence} gives an
 equivariant $*$-isomorphism
 $F:\A_{\widehat\gamma}^\omega
      \longrightarrow\A_{\widetilde\gamma}^\omega$.
 Define
 \[
 \tau=\Phi_{\widetilde\gamma}^{-1}
          F\Phi_{\widehat\gamma}:\A_\gamma\longrightarrow\A_\gamma.
 \]
 Then $\tau$ is an automorphism. The equivariance of $F$ gives
 \[
 \begin{aligned}
 \tau\circ\delta_C^{\widehat\gamma}
   &=\Phi_{\widetilde\gamma}^{-1}\circ F\circ
       \lambda_C^{\widehat\gamma}\circ\Phi_{\widehat\gamma}=\delta_C^{\widetilde\gamma}\circ\tau
       \qquad(C\in\mathrm{SL}_2(R_p)),
 \end{aligned}
 \]
 so the transported actions are conjugate.

 The conjugacy can be made explicit. If $p$ is even,
 the proof of Proposition~\ref{prop:lift-independence}
 shows that $h_{\widehat\gamma}=h_{\widetilde\gamma}$.
 Consequently, the multipliers and cochains coincide, and the
 canonical identification gives $\tau=\id$.
 If $p$ is odd, let $\varepsilon\in\{0,1\}$ be
 the common parity of
 $\widehat\gamma_n-\widetilde\gamma_n$, and use reduction modulo
 $2$ on $R_p$ as in that proposition.
 Its isomorphism has the form
 \[
 F(W_{(a,b)}^{\widehat\gamma})
   =(-1)^{\varepsilon(\bar a\bar b+\bar a+\bar b)}
        W_{(a,b)}^{\widetilde\gamma}.
 \]
 Since
 $c_{\widehat\gamma}(a,b)/c_{\widetilde\gamma}(a,b)
    =(-1)^{\varepsilon\bar a\bar b}$, we obtain
 \begin{equation}\label{eq:standard-lift-conjugacy}
 \tau(W_{(a,b)}^\Psi)
   =(-1)^{\varepsilon(\bar a+\bar b)}W_{(a,b)}^\Psi.
 \end{equation}
 Thus $\tau$ acts on the canonical unitaries by a character of
 $\Gamma_p$.

 Let $u_C^{\widehat\gamma}$ and $u_C^{\widetilde\gamma}$ denote
 the implementing unitaries in the two standard crossed products.
 The assignments
 \[
 x\longmapsto\tau(x),\qquad
 u_C^{\widehat\gamma}\longmapsto u_C^{\widetilde\gamma}
       \quad(x\in\A_\gamma,\ C\in H)
 \]
 preserve covariance, since
 \[
 u_C^{\widetilde\gamma}\tau(x)(u_C^{\widetilde\gamma})^*
   =\delta_C^{\widetilde\gamma}(\tau(x))
   =\tau(\delta_C^{\widehat\gamma}(x)).
 \]
 By functoriality of the reduced crossed product
 \cite{Phi17}*{Theorem~9.23}, these assignments extend to a
 $*$-homomorphism. The map induced by $\tau^{-1}$ is its inverse.
 This proves the first crossed-product isomorphism.
 Similarly, $\Phi_{\widetilde\gamma}$ induces the second
 isomorphism, acting by $\Phi_{\widetilde\gamma}$ on coefficients
 and sending each $u_C^{\widetilde\gamma}$ to the corresponding
 implementing unitary of the alternating crossed product.
\end{proof}

For odd $p$, the sequence $\widehat\gamma_n=\widetilde\gamma_n+1$
is another even lift of the same parameter. For
$L=\begin{psmallmatrix}1&0\\1&1\end{psmallmatrix}$,
\eqref{eq:transported-matrix-action} gives
\[
 \delta_L^{\widehat\gamma}(W_{(1,0)}^\Psi)
   =-\delta_L^{\widetilde\gamma}(W_{(1,0)}^\Psi).
\]
Thus the generator formulas can depend on the even lift even though
the resulting crossed products are isomorphic.

For $C=-I$, the phase in
\eqref{eq:transported-matrix-action} is $1$, so the transported
action coincides with the standard flip.

\begin{cor}\label{cor:standard-alternating-flip}
 Let $\alpha\in\Xi_p$, and choose any even real lift
 $\widetilde\alpha\in\Omega_p^{\mathrm{even}}$ with
 $h(\widetilde\alpha)=\alpha$. Then the isomorphism
 $\Phi_{\widetilde\alpha}:\A_\alpha\to
 \A_{\widetilde\alpha}^\omega$ in
 \eqref{eq:Phi-standard-to-alternating} is equivariant for
 the standard flip and $\lambda_{-I}^{\widetilde\alpha}$:
 \[
 \Phi_{\widetilde\alpha}\circ\mathrm{flip}_\alpha
   =\lambda_{-I}^{\widetilde\alpha}
      \circ\Phi_{\widetilde\alpha}.
 \]
 Consequently, it induces a $*$-isomorphism
 \[
 \A_\alpha\rtimes_{\mathrm{flip}_\alpha}(\Z/2\Z)
   \ \cong\
 \A_{\widetilde\alpha}^\omega
       \rtimes_{\lambda_{-I}^{\widetilde\alpha}}(\Z/2\Z).
 \]
\end{cor}
\begin{proof}
 For $z=(a,b)$, the identity
 $c_{\widetilde\alpha}(-z)=c_{\widetilde\alpha}(z)$ gives
 \[
 \begin{aligned}
 \Phi_{\widetilde\alpha}(\mathrm{flip}_\alpha(W_z^\Psi))
   &=c_{\widetilde\alpha}(-z)W_{-z}=c_{\widetilde\alpha}(z)W_{-z}=\lambda_{-I}^{\widetilde\alpha}
       (\Phi_{\widetilde\alpha}(W_z^\Psi)).
 \end{aligned}
 \]
 This proves equivariance on the generating unitaries and hence
 on $\A_\alpha$. Since $\Z/2\Z$ is amenable, the universal property
 gives the crossed-product isomorphism, sending the implementing unitary of the standard
 flip to that of $\lambda_{-I}^{\widetilde\alpha}$.
\end{proof}

The following realization of a crossed product as a twisted group
algebra is an application of \cite{ELPW10}*{Lemma~2.1}.

\begin{prp}\label{prop:semidirect-crossed-product}
 Let $\alpha\in\Omega_p^{\mathrm{even}}$ and
 $H\subseteq\mathrm{SL}_2(R_p)$ be a subgroup. Give
 $\Gamma_p\rtimes H$ the multiplication
 \[
 (z,A)(w,B)=(z+Aw,AB).
 \]
 Then the map $\widetilde\omega_\alpha$ defined by $\widetilde\omega_\alpha((z,A),(w,B))
    =\omega_\alpha(z,Aw)$ is a normalized multiplier on $\Gamma_p\rtimes H$ and we obtain
 \begin{equation}\label{eq:semidirect-crossed-product}
 \A_\alpha^\omega\rtimes_{\lambda^\alpha}H
    \cong C^*(\Gamma_p\rtimes H,\widetilde\omega_\alpha)
 \end{equation}
 for the reduced completions.
\end{prp}
\begin{proof}
 By Proposition~\ref{prop:invariant-multiplier}, $\omega_\alpha$
 is invariant under $H$. Since $\Gamma_p$ and $H$ are countable
 discrete groups, the reduced form of the cited lemma applies
 with $M=\Gamma_p$ and $\omega=\omega_\alpha$, giving the stated multiplier and
 isomorphism. Its induced action sends $W_z$ to $W_{Az}$
 and hence agrees with $\lambda^\alpha$.
\end{proof}

If $H$ is amenable, then so is $\Gamma_p\rtimes H$, and both
sides of \eqref{eq:semidirect-crossed-product} agree with their full
completions. This applies, in particular, to cyclic groups.

\subsection{The \texorpdfstring{$\mathrm{SL}_2(\Z)$}{SL2(Z)} action and inductive limits}
\label{inductive}

Fix $\alpha\in\Omega_p^{\mathrm{even}}$. For each $n\geq0$, the
rotation algebra $A_{\alpha_{2n}}$ has canonical generators
$V_{2n},U_{2n}$ satisfying
$V_{2n}U_{2n}=e(\alpha_{2n})U_{2n}V_{2n}$.
Recall the Watatani--Brenken action
$\pi^{(n)}:\mathrm{SL}_2(\Z)\to\Aut(A_{\alpha_{2n}})$
\cites{Wat,Bre84}. In our generator convention, for
$A=\begin{psmallmatrix}a&b\\c&d\end{psmallmatrix}\in\mathrm{SL}_2(\Z)$,
\begin{align}
 \pi_A^{(n)}(V_{2n})
   &=e(-ac\alpha_{2n}/2)V_{2n}^{a}U_{2n}^{c},
                       \label{eq:WB-V}\\
 \pi_A^{(n)}(U_{2n})
   &=e(-bd\alpha_{2n}/2)V_{2n}^{b}U_{2n}^{d}.
                       \label{eq:WB-U}
\end{align}
In the Weyl presentation \eqref{eq:rotation-Weyl-convention}, this
action is induced by the natural matrix action on $\Z^2$.

We use the connecting maps $\varphi_n$ of \eqref{connecting mor}.
Theorem~\ref{thm:solenoid-inductive-limit} and $\Phi_\alpha$ identify
their inductive limit with $\A_\alpha^\omega$. Write
$\iota_n=\Phi_\alpha\circ\nu_n$ for the canonical stage embeddings.
Since $c_\alpha$ is $1$ on the coordinate axes,
\[
 \iota_n(V_{2n})=W_{(p^{-n},0)},\qquad
 \iota_n(U_{2n})=W_{(0,p^{-n})}.
\]
The following proposition compares the action induced by the
rotation-algebra stages with the matrix action on the solenoid.

\begin{prp}\label{compat of inductive}\label{equal action}
 Let $\alpha\in\Omega_p^{\mathrm{even}}$. The actions $\pi^{(n)}$
 are compatible with the connecting maps:
 \[
 \pi_A^{(n+1)}\circ\varphi_n
   =\varphi_n\circ\pi_A^{(n)}
       \qquad(A\in\mathrm{SL}_2(\Z),\ n\geq0).
 \]
 They therefore induce an action
 $\Lambda:\mathrm{SL}_2(\Z)\to\Aut(\A_\alpha^\omega)$.
 This action coincides with the restriction of $\lambda^\alpha$
 to $\mathrm{SL}_2(\Z)$:
 \[
 \Lambda_A=\lambda_A^\alpha
       \qquad(A\in\mathrm{SL}_2(\Z)).
 \]
\end{prp}
\begin{proof}
 Fix $A=\begin{psmallmatrix}a&b\\c&d\end{psmallmatrix}
 \in\mathrm{SL}_2(\Z)$. To verify compatibility, note that if
 $VU=e(\gamma)UV$, then for $m\geq1$,
 \[
 (V^aU^c)^m
  =e\!\left(-\frac{ac\gamma m(m-1)}2\right)
      V^{ma}U^{mc}.
 \]
 Using \eqref{eq:WB-V} and $\varphi_n(V_{2n})=V_{2n+2}^p$, we obtain
 \[
 \begin{aligned}
 (\pi_A^{(n+1)}\circ\varphi_n)(V_{2n})
   &=e(-acp^2\alpha_{2n+2}/2)V_{2n+2}^{pa}U_{2n+2}^{pc},\\
 (\varphi_n\circ\pi_A^{(n)})(V_{2n})
   &=e(-ac\alpha_{2n}/2)V_{2n+2}^{pa}U_{2n+2}^{pc}.
 \end{aligned}
 \]
 These phases agree because
 \[
 p^2\alpha_{2n+2}-\alpha_{2n}
  =p(p\alpha_{2n+2}-\alpha_{2n+1})
    +(p\alpha_{2n+1}-\alpha_{2n})\in2\Z.
 \]
 The same calculation with $(b,d)$ gives
 compatibility on $U_{2n}$, and hence on the whole stage algebra.
 By \cite{Phi17}*{Proposition~3.24}, these compatible actions induce
 a unique action $\Lambda$ on $\A_\alpha^\omega$ satisfying
 \[
 \Lambda_A\circ\iota_n=\iota_n\circ\pi_A^{(n)}.
 \]

 To compare this action with $\lambda^\alpha$, the Weyl convention
 \eqref{eq:rotation-Weyl-convention} gives
 \begin{equation}\label{eq:stage-Weyl-monomial}
 W_{(r/p^n,s/p^n)}
   =e(-rs\alpha_{2n}/2)\,
      \iota_n(V_{2n}^{r}U_{2n}^{s})
       \qquad(r,s\in\Z).
 \end{equation}
 Thus we have
 \[
 \begin{aligned}
 \lambda_A^\alpha(\iota_n(V_{2n}))=W_{(a/p^n,c/p^n)}=\iota_n(\pi_A^{(n)}(V_{2n})).
 \end{aligned}
 \]
 The same argument on $U_{2n}$ gives
 $\lambda_A^\alpha\circ\iota_n=\iota_n\circ\pi_A^{(n)}$.
 Consequently, $\Lambda_A$ and $\lambda_A^\alpha$ agree on each
 stage image $\iota_n(A_{\alpha_{2n}})$. The union of these stage
 images is dense in $\A_\alpha^\omega$, so
 $\Lambda_A=\lambda_A^\alpha$ on the solenoid.
\end{proof}

\begin{cor}\label{cor:crossed-product-inductive-limit}
 Let $H\subseteq\mathrm{SL}_2(\Z)$ be a subgroup and let
 $\alpha\in\Omega_p^{\mathrm{even}}$. Then
 \[
 \A_\alpha^\omega\rtimes_{\lambda^\alpha}H
  \cong\varinjlim_n
   \bigl(A_{\alpha_{2n}}\rtimes_{\pi^{(n)}}H\bigr),
 \]
 with connecting maps induced by $\varphi_n$ on
 coefficients and the identity on $H$.
\end{cor}
\begin{proof}
 Write $D=\A_\alpha^\omega$, $A_n=A_{\alpha_{2n}}$, and
 $B_n=A_n\rtimes_{\pi^{(n)}}H$. Denote by $j_n:A_n\to B_n$
 and $j:D\to D\rtimes_{\lambda^\alpha}H$ the canonical coefficient
 maps, and by $u_A^{(n)}$ and $u_A$, respectively, the canonical
 group unitaries, for $A\in H$.

 By Proposition~\ref{compat of inductive}, each $\varphi_n$ is
 equivariant. Since $\varphi_n$ is also injective, functoriality
 and injectivity of the reduced crossed product
 \cite{Phi17}*{Theorems~9.23--9.24} give an injective
 $*$-homomorphism $\psi_n:B_n\to B_{n+1}$ satisfying
 \[
 \psi_n(j_n(x))=j_{n+1}(\varphi_n(x)),\qquad
 \psi_n(u_A^{(n)})=u_A^{(n+1)}.
 \]
 Similarly, the stage embeddings $\iota_n$ are equivariant and
 injective. They induce injective $*$-homomorphisms
 \[
 G_n:B_n\longrightarrow D\rtimes_{\lambda^\alpha}H
 \]
 satisfying
 \[
 G_n(j_n(x))=j(\iota_n(x)),\qquad
 G_n(u_A^{(n)})=u_A.
 \]
 The identity $\iota_{n+1}\circ\varphi_n=\iota_n$ gives
 $G_{n+1}\circ\psi_n=G_n$. Hence these maps induce
 \[
 G:\varinjlim(B_n,\psi_n)\longrightarrow
       D\rtimes_{\lambda^\alpha}H.
 \]
 Each $G_n$ is isometric, so $G$ is isometric on the dense union
 of the stage images and therefore injective.

 The range contains every $u_A$ and all coefficient elements
 $j(\iota_n(x))$. Since the union of the stage images
 $\iota_n(A_n)$ is dense in $D$, these elements generate
 $D\rtimes_{\lambda^\alpha}H$. The range of a
 $*$-homomorphism between $\mathrm C^*$-algebras is closed,
 so $G$ is surjective. This proves the asserted isomorphism.
\end{proof}

%% file: bimodule_revised.tex
\section{The reciprocal Heisenberg equivalence bimodule}
\label{bimodule}

We describe the Heisenberg bimodule realizing the Morita equivalence
between $\A_\alpha$ and $\A_\beta$, where $\beta$ is the reciprocal
parameter defined in \eqref{eq:reciprocal-parameters}. Following
\cite{Rie88}*{Sections~2--3}, we give the two lattices, the module
actions, and the inner products on $\mathcal S(\Q_p\times\R)$
in the alternating models of Section~\ref{crossed product}.
We choose the even lifts explicitly to obtain the Weyl formulas
used in the equivariance arguments of the following sections.
Throughout this section, $p$ is an odd prime.

\subsection{Parameters, Haar measure, and Weyl operators}
Fix $\alpha\in\Xi_p$ such that
\[
 \theta=\alpha_0\in(0,1),\qquad
 x=x_\alpha=\sum_{j\geq0}x_jp^j\in\Z_p\setminus\{0\},
 \qquad x_j=p\alpha_{j+1}-\alpha_j\in\{0,\ldots,p-1\}.
\]
Here $x_j$ are the transition integers of $\alpha$.
Put $v=v_p(x)\geq0$, and write the canonical Laurent expansion
\[
 y=x^{-1}=\sum_{j=-v}^{\infty}y_jp^j\in\Q_p,
 \qquad y_j\in\{0,\ldots,p-1\}.
\]
For $n\geq0$, set
\[
 X_n=\sum_{j=0}^{n-1}x_jp^j,\qquad
 Y_n=\sum_{j=-v}^{n-1}y_jp^j.
\]
The transition relation for $\alpha$ gives the first formula below.
Following \cite{Lu22}*{Theorem~3.6}, we define the reciprocal
real parameter $\beta$ by the second:
\begin{equation}\label{eq:reciprocal-parameters}
 \alpha_n=\frac{\theta+X_n}{p^n},
 \qquad
 \beta_n=\frac{\theta^{-1}+Y_n}{p^n}
 \qquad(n\geq0).
\end{equation}
Since $p\beta_{n+1}-\beta_n=y_n$, we have $\beta\in\Omega_p$;
we retain this real sequence without reducing its entries modulo
$\Z$. When $v>0$, $Y_n$ includes the negative powers of $p$
in the expansion of $y$.

For $r=j/p^n\in R_p$, we have
$(x-X_n)r,(y-Y_n)r\in\Z_p$.
Since $e_p$ is trivial on $\Z_p$ and agrees with $e$ on $R_p$,
\eqref{eq:reciprocal-parameters} gives
\begin{equation}\label{eq:reciprocal-standard-characters}
 e_p(xr)e(\theta r)=e(\alpha_nj),
 \qquad
 e_p(yr)e(r/\theta)=e(\beta_nj).
\end{equation}
These identities determine the standard multipliers on the two
lattices below.

The Weyl multiplication law involves the factors
$e_p(xr/2)e(\theta r/2)$ and $e_p(yr/2)e(r/(2\theta))$.
To express them as $h_{\widetilde\alpha}(r)$ and
$h_{\widetilde\beta}(r)$, we choose even lifts using the
$p$-adic expansions of $x/2$ and $y/2$. Since $p$ is odd,
$2$ is a unit in $\Z_p$, and these expansions have the form
\[
 x/2=\sum_{j\geq0}t_jp^j\in\Z_p,\qquad
 y/2=\sum_{j=-v}^{\infty}s_jp^j\in\Q_p,
 \qquad t_j,s_j\in\{0,\ldots,p-1\}.
\]
Set
\[
 T_n=\sum_{j=0}^{n-1}t_jp^j,\qquad
 S_n=\sum_{j=-v}^{n-1}s_jp^j,
\]
and define
\begin{equation}\label{eq:reciprocal-even-lifts}
 \widetilde\alpha_n=\frac{\theta+2T_n}{p^n},
 \qquad
 \widetilde\beta_n=\frac{\theta^{-1}+2S_n}{p^n}.
\end{equation}
Their transition integers are even, since
\begin{equation}\label{eq:reciprocal-even-transitions}
 p\widetilde\alpha_{n+1}-\widetilde\alpha_n=2t_n,
 \qquad
 p\widetilde\beta_{n+1}-\widetilde\beta_n=2s_n.
\end{equation}
Since $X_n$ and $2T_n$ both agree with $x$ modulo $p^n\Z_p$,
their difference belongs to $p^n\Z_p\cap R_p=p^n\Z$.
The same argument applies to $Y_n$ and $2S_n$. Hence
\begin{equation}\label{eq:reciprocal-lift-reductions}
 \widetilde\alpha_n-\alpha_n
   =\frac{2T_n-X_n}{p^n}\in\Z,
 \qquad
 \widetilde\beta_n-\beta_n
   =\frac{2S_n-Y_n}{p^n}\in\Z.
\end{equation}
Consequently, $\widetilde\alpha,\widetilde\beta\in
\Omega_p^{\mathrm{even}}$ have the same coordinatewise reductions
as $\alpha$ and $\beta$, respectively. In particular,
\[
 h(\widetilde\alpha)=\alpha,\qquad
 h(\widetilde\beta)=h(\beta),\qquad
 \Psi_{\widetilde\alpha}=\Psi_\alpha,\qquad
 \Psi_{\widetilde\beta}=\Psi_\beta.
\]
We call these the \emph{pairing-compatible even lifts}, since
their characters agree with the Weyl phases computed in
Lemma~\ref{lem:reciprocal-half-phases}.

Let $M=\Q_p\times\R$. We identify $M$ with its Pontryagin dual
using the symmetric bicharacter
\begin{equation}\label{eq:heisenberg-pairing}
 \chi((q,t),(q',t'))
    =e_p(yqq')e(tt'/\theta).
\end{equation}
Multiplication by $y$ on $\Q_p$ and by $\theta^{-1}$ on $\R$
gives duality isomorphisms, so $\chi$ is nondegenerate.
It also satisfies
\begin{equation}\label{eq:heisenberg-balance}
 \chi(ru,v)=\chi(u,rv)
       \qquad(r\in R_p,\ u,v\in M).
\end{equation}
The self-dual Haar measure for this pairing is
\begin{equation}\label{eq:heisenberg-measure}
 d\mu(q,t)=|x|_p^{-1/2}\theta^{-1/2}\,dq\,dt,
        \qquad dq(\Z_p)=1.
\end{equation}
Indeed, the $p$-adic factor has self-dual measure
$|y|_p^{1/2}dq=|x|_p^{-1/2}dq$, and the real factor has
self-dual measure $\theta^{-1/2}dt$. We use the scalar inner
product
\[
 (f,g)_2=\int_M f(s)\overline{g(s)}\,d\mu(s),
\]
which is linear in the first variable. The module operations
will first be defined on the Schwartz space $\mathcal S(M)$.

Put $G=M\times M$. For $z=(u,v)\in G$, define the Weyl
operator $\mathcal W(z)=\mathcal W(u,v)$ on $L^2(M,\mu)$ by
\begin{equation}\label{eq:heisenberg-Weyl-operator}
 (\mathcal W(u,v)f)(s)
       =\chi(s-u/2,v)f(s-u).
\end{equation}
Each $\mathcal W(z)$ is unitary and preserves $\mathcal S(M)$,
since it is a translation followed by multiplication by a
character and a scalar of modulus one. We record its
multiplication law for use in the module actions.

\begin{lem}\label{lem:heisenberg-Weyl-product}
 For $z=(u,v)$ and $w=(u',v')$ in $G$, put $m(z,w)=\chi(u'/2,v)\overline{\chi(u/2,v')}$.
 Then $m$ is a continuous multiplier on $G$, and
 \[
 \mathcal W(z)\mathcal W(w)=m(z,w)\mathcal W(z+w),
 \qquad
 \mathcal W(z)^*=\mathcal W(-z).
 \]
 The commutator pairing is
 \begin{equation}\label{eq:heisenberg-commutator}
 \sigma(z,w)
   =\frac{m(z,w)}{m(w,z)}
   =\chi(u',v)\overline{\chi(u,v')},
 \end{equation}
 and is nondegenerate.
\end{lem}
\begin{proof}
 Since $\chi$ is a continuous bicharacter, so is $m$; in
 particular, $m$ is a normalized $2$-cocycle. For
 $f\in\mathcal S(M)$, direct computation gives
 \[
 \begin{aligned}
  (\mathcal W(z)\mathcal W(w)f)(s)
   &=\chi(s-u/2,v)\chi(s-u-u'/2,v')f(s-u-u')\\
   &=m(z,w)\chi(s-(u+u')/2,v+v')f(s-u-u')\\
   &=m(z,w)(\mathcal W(z+w)f)(s).
 \end{aligned}
 \]
 By density, the product formula holds on $L^2(M,\mu)$.
 Since $m(z,-z)=1$, the operator $\mathcal W(-z)$ is the
 inverse of $\mathcal W(z)$, proving the adjoint formula.
 Taking the ratio $m(z,w)/m(w,z)$ gives
 \eqref{eq:heisenberg-commutator}.

 If $\sigma((u,v),(u',v'))=1$ for every $(u',v')\in G$,
 setting $v'=0$ gives $\chi(u',v)=1$ for every $u'\in M$,
 so $v=0$. Setting $u'=0$ similarly gives $u=0$.
 Hence $\sigma$ is nondegenerate.
\end{proof}

\subsection{The two lattices and their annihilator relation}

Define the $R_p$-linear maps $\rho,\ell:R_p\to M=\Q_p\times\R$ by
\[
 \rho(a)=(xa,\theta a),\qquad
 \ell(a)=(a,-a),\qquad a\in R_p.
\]
They determine embeddings
$\iota_\alpha,\iota_\beta:\Gamma_p=R_p^2\to G=M\times M$ given by
\begin{equation}\label{eq:heisenberg-lattice-embeddings}
 \iota_\alpha(a,b)=(-\rho(a),-\rho(b)),\qquad
 \iota_\beta(a,b)=(\ell(a),-\ell(b)).
\end{equation}
Put $D_R=\iota_\alpha(\Gamma_p)$ and
$D_L=\iota_\beta(\Gamma_p)$.
For $z=(a,b)$ and $w=(c,d)$ in $\Gamma_p$, the pairings
$\chi(\rho(a),\rho(d))$ and $\chi(\ell(a),\ell(d))$ are
$e_p(xad)e(\theta ad)$ and $e_p(yad)e(ad/\theta)$, respectively.
Thus \eqref{eq:reciprocal-standard-characters}, with $r=ad$, gives
\begin{equation}\label{eq:pairing-standard-multipliers}
 \chi(\rho(a),\rho(d))=\Psi_\alpha(z,w),\qquad
 \chi(\ell(a),\ell(d))=\Psi_\beta(z,w).
\end{equation}
For a subgroup $D\subseteq G$, its annihilator with respect
to the commutator pairing $\sigma$ is
\[
 D^\perp
    =\{z\in G:\sigma(z,d)=1\text{ for every }d\in D\}.
\]
By a lattice we mean a discrete subgroup with compact quotient.

\begin{lem}\label{lem:heisenberg-lattices}
 The groups $D_R,D_L$ are lattices in $G$, and
 \[
 D_R^\perp=D_L,\qquad D_L^\perp=D_R.
 \]
 For the measure $\mu\times\mu$, their covolumes are
 \begin{equation}\label{eq:heisenberg-covolumes}
 \operatorname{covol}(D_R)=|x|_p\theta,\qquad
 \operatorname{covol}(D_L)=(|x|_p\theta)^{-1}.
 \end{equation}
\end{lem}
\begin{proof}
 By \cite{JP13}*{Section~5.1}, $\ell(R_p)$ is a lattice in $M$.
 Reflection in the real coordinate gives the diagonal lattice
 $\Delta R_p=\{(r,r):r\in R_p\}$, with fundamental domain
 $\Z_p\times[0,1)$ of measure $1$ for $dq\,dt$.
 The automorphism $(q,t)\mapsto(xq,\theta t)$ sends
 $\Delta R_p$ to $\rho(R_p)$. Hence
 $D_R=\rho(R_p)\times\rho(R_p)$ and
 $D_L=\ell(R_p)\times\ell(R_p)$ are lattices.

 The scaling multiplies Haar measure by $|x|_p\theta$,
 while reflection preserves it. Using the normalization in
 \eqref{eq:heisenberg-measure}, we obtain
 \[
 \operatorname{covol}_\mu(\rho(R_p))=(|x|_p\theta)^{1/2},
 \qquad
 \operatorname{covol}_\mu(\ell(R_p))=(|x|_p\theta)^{-1/2}.
 \]
 Taking products gives \eqref{eq:heisenberg-covolumes}.

 For the annihilator relation, let
 $\chi_0((q,t),(q',t'))=e_p(qq')e(tt')$ be the standard
 pairing used in \cite{JP13}. Define continuous automorphisms
 $C:M\to M$ and $F:G\to G$ by
 \[
 C(q,t)=(yq,t/\theta),\qquad F(u,v)=(u,Cv).
 \]
 We have $\chi(u,v)=\chi_0(u,Cv)$. Thus
 \eqref{eq:heisenberg-commutator} shows that $F$ carries our
 commutator annihilators to those for $\chi_0$.
 Changing signs of the indexing variables gives
 \[
 \begin{aligned}
 F(D_R)&=\{((xa,\theta a),(b,b)):a,b\in R_p\},\\
 F(D_L)&=\{((a,-a),(yb,-b/\theta)):a,b\in R_p\}.
 \end{aligned}
 \]
 Since $x\ne0$ and $\theta\ne0$,
 \cite{JP13}*{Lemma~5.5} identifies the second lattice as the
 annihilator of the first. Applying $F^{-1}$ gives
 $D_R^\perp=D_L$. Finally, $D_R$ is closed, so the
 double-annihilator theorem gives $D_L^\perp=D_R$.
\end{proof}

\subsection{The left and right module actions}

The annihilator relation gives commuting Weyl operators on the
two lattices. We now identify their multipliers to define the
module actions.

Let $W_z$ and $M_z$ denote the canonical unitaries of
$\A_{\widetilde\alpha}^\omega
=C^*(\Gamma_p,\omega_{\widetilde\alpha})$ and
$\A_{\widetilde\beta}^\omega
=C^*(\Gamma_p,\omega_{\widetilde\beta})$, respectively.
For the lifts in \eqref{eq:reciprocal-even-lifts}, set
\[
 R_z=\mathcal W(\iota_\alpha(z)),\qquad
 L_z=\mathcal W(\iota_\beta(z)).
\]
These operators preserve $\mathcal S(M)$. For $f\in\mathcal S(M)$,
define
\[
 f\cdot W_z=R_zf,\qquad M_z\cdot f=L_zf,
\]
and extend by linearity to finite sums of the canonical unitaries.
The following identities determine the scalar factors in their
Weyl products.

\begin{lem}\label{lem:reciprocal-half-phases}
 For every $r\in R_p$,
 \begin{equation}\label{eq:reciprocal-half-phases}
 e_p(xr/2)e(\theta r/2)=h_{\widetilde\alpha}(r),
 \qquad
 e_p(yr/2)e(r/(2\theta))=h_{\widetilde\beta}(r).
 \end{equation}
\end{lem}
\begin{proof}
 Write $r=j/p^k$. The truncations satisfy
 $x/2-T_k,y/2-S_k\in p^k\Z_p$.
 Since $e_p$ is trivial on $\Z_p$ and agrees with $e$ on $R_p$,
 \[
 e_p(xr/2)=e(T_kj/p^k),\qquad
 e_p(yr/2)=e(S_kj/p^k).
 \]
 Multiplying by the real factors and using
 \eqref{eq:reciprocal-even-lifts} gives
 \[
 \begin{aligned}
 e_p(xr/2)e(\theta r/2)
   &=e\left(\frac{(2T_k+\theta)j}{2p^k}\right)
     =e(\widetilde\alpha_kj/2),\\
 e_p(yr/2)e(r/(2\theta))
   &=e\left(\frac{(2S_k+\theta^{-1})j}{2p^k}\right)
     =e(\widetilde\beta_kj/2).
 \end{aligned}
 \]
 These are $h_{\widetilde\alpha}(r)$ and
 $h_{\widetilde\beta}(r)$ by definition.
\end{proof}

\begin{prp}\label{prop:heisenberg-module-actions}
 For $z,w\in\Gamma_p$,
 \begin{equation}\label{eq:heisenberg-action-products}
 R_zR_w=\overline{\omega_{\widetilde\alpha}(z,w)}R_{z+w},
 \qquad
 L_zL_w=\omega_{\widetilde\beta}(z,w)L_{z+w},
 \qquad
 R_zL_w=L_wR_z.
 \end{equation}
 Consequently, the formulas above define commuting
 right and left actions of the algebraic twisted
 group algebras on $\mathcal S(M)$.
\end{prp}
\begin{proof}
 Let $z=(a,b)$ and $w=(c,d)$. By
 Lemma~\ref{lem:reciprocal-half-phases}, the restriction of the
 Weyl multiplier to the right lattice satisfies
 \[
 \begin{aligned}
 m(\iota_\alpha(z),\iota_\alpha(w))
   &=\chi(\rho(c)/2,\rho(b))
       \overline{\chi(\rho(a)/2,\rho(d))}=h_{\widetilde\alpha}(cb)
       \overline{h_{\widetilde\alpha}(ad)}
     =\overline{\omega_{\widetilde\alpha}(z,w)}.
 \end{aligned}
 \]
 Similarly, on the left lattice,
 \[
 \begin{aligned}
 m(\iota_\beta(z),\iota_\beta(w))
   &=\overline{\chi(\ell(c)/2,\ell(b))}
       \chi(\ell(a)/2,\ell(d))=\overline{h_{\widetilde\beta}(cb)}
       h_{\widetilde\beta}(ad)
     =\omega_{\widetilde\beta}(z,w).
 \end{aligned}
 \]
 Lemma~\ref{lem:heisenberg-Weyl-product} therefore gives the
 first two identities in \eqref{eq:heisenberg-action-products}.
 Since $D_L=D_R^\perp$, the commutator relation gives
 \[
 R_zL_w
   =\sigma(\iota_\alpha(z),\iota_\beta(w))L_wR_z
   =L_wR_z.
 \]

 For the right action, the order of the operators is reversed:
 \[
 (f\cdot W_z)\cdot W_w
   =R_wR_zf
   =\omega_{\widetilde\alpha}(z,w)R_{z+w}f
   =f\cdot(W_zW_w),
 \]
 where alternation of $\omega_{\widetilde\alpha}$ was used.
 The identity for $L_zL_w$ proves associativity of the left
 action, and $R_zL_w=L_wR_z$ proves that the two actions commute.
\end{proof}

Substituting the lattice embeddings into the Weyl operators
gives the following coordinate formulas.

\begin{prp}\label{prop:heisenberg-coordinate-actions}
 Let $a=j_1/p^{k_1}$ and $b=j_2/p^{k_2}$, where
 $j_1,j_2\in\Z$ and $k_1,k_2\geq0$. For
 $f\in\mathcal S(M)$ and $(q,t)\in M$, one has
 \begin{align}
 (f\cdot W_{a,b})(q,t)
  &=e(-\widetilde\alpha_{k_1+k_2}j_1j_2/2)
      e_p(-qb)e(-tb)f(q+xa,t+\theta a),
       \label{eq:right-coordinate-action}\\
 (M_{a,b}\cdot f)(q,t)
  &=e(\widetilde\beta_{k_1+k_2}j_1j_2/2)
      e_p(-yqb)e(tb/\theta)f(q-a,t+a).
       \label{eq:left-coordinate-action}
 \end{align}
\end{prp}
\begin{proof}
 Substituting $u=-\rho(a)$ and $v=-\rho(b)$ into
 \eqref{eq:heisenberg-Weyl-operator} gives
 \[
 \begin{aligned}
 (R_{a,b}f)(q,t)
   &=e_p(-qb-xab/2)e(-tb-\theta ab/2)
       f(q+xa,t+\theta a)\\
   &=\overline{h_{\widetilde\alpha}(ab)}
       e_p(-qb)e(-tb)f(q+xa,t+\theta a).
 \end{aligned}
 \]
 The last equality uses Lemma~\ref{lem:reciprocal-half-phases}.
 Substituting $u=\ell(a)$ and $v=-\ell(b)$ similarly gives
 \[
 \begin{aligned}
 (L_{a,b}f)(q,t)
   &=e_p(-yqb+yab/2)e(tb/\theta+ab/(2\theta))
       f(q-a,t+a)\\
   &=h_{\widetilde\beta}(ab)
       e_p(-yqb)e(tb/\theta)f(q-a,t+a).
 \end{aligned}
 \]
 Since $ab=j_1j_2/p^{k_1+k_2}$, the definitions of the two
 characters give
 \eqref{eq:right-coordinate-action} and
 \eqref{eq:left-coordinate-action}.
\end{proof}

\subsection{Inner products and completion}

We now define the algebra-valued inner products and complete
$\mathcal S(M)$ to an imprimitivity bimodule. We use counting
measure on both lattices.

For $N\geq0$, let $\mathscr S(p^{-N}\Z^2)$ be the space of
sequences rapidly decreasing in the Euclidean norm, extended
by zero to $\Gamma_p$. With twisted convolution and involution
for the restriction of $\omega$, this space is denoted by
$\mathscr S(p^{-N}\Z^2,\omega)$. Set
\[
 \mathcal A_{\widetilde\alpha}^0
    =\bigcup_{N\geq0}
       \mathscr S(p^{-N}\Z^2,\omega_{\widetilde\alpha}),
 \qquad
 \mathcal A_{\widetilde\beta}^0
    =\bigcup_{N\geq0}
       \mathscr S(p^{-N}\Z^2,\omega_{\widetilde\beta}).
\]
Twisted convolution and involution preserve rapid decrease on
each $p^{-N}\Z^2$. Thus these are $*$-subalgebras of the
corresponding solenoids. They are dense because they contain
all finitely supported sequences on $\Gamma_p$.
For $f\in\mathcal S(M)$, extend the module actions by
\[
 f\cdot\left(\sum_z a_zW_z\right)=\sum_z a_zR_zf,
 \qquad
 \left(\sum_z b_zM_z\right)\cdot f=\sum_z b_zL_zf.
\]
Both series converge in $\mathcal S(M)$. On a fixed
$p^{-N}\Z^2$, the $p$-adic factors have a common compact
support and a common subgroup of local constancy, so they lie
in a finite-dimensional space. The real Schwartz seminorms
of $R_zf$ and $L_zf$ grow at most polynomially in $z$.
Rapid decrease of the coefficients therefore gives convergence.

For $f,g\in\mathcal S(M)$, define
\begin{align}
 \langle f,g\rangle_{\A_{\widetilde\alpha}^\omega}
   &=\sum_{z\in\Gamma_p}(R_{-z}g,f)_2\,W_z,
                      \label{eq:heisenberg-right-inner-product}\\
 {}_{\A_{\widetilde\beta}^\omega}\langle f,g\rangle
   &=(|x|_p\theta)^{-1}
       \sum_{z\in\Gamma_p}(f,L_zg)_2\,M_z.
                      \label{eq:heisenberg-left-inner-product}
\end{align}
Here $R_{-z}=R_z^*$. Since $(\cdot,\cdot)_2$ is linear in
its first variable, the right inner product is linear in $g$
and the left inner product is linear in $f$.
The factor $(|x|_p\theta)^{-1}=\operatorname{covol}(D_R)^{-1}$
is required by Poisson summation when both sums use counting
measure. The next lemma places these inner products in the
dense algebras above.

\begin{lem}\label{lem:heisenberg-coefficient-decay}
 For $f,g\in\mathcal S(M)$, the coefficient sequences
 in \eqref{eq:heisenberg-right-inner-product} and
 \eqref{eq:heisenberg-left-inner-product} belong to
 $\mathcal A_{\widetilde\alpha}^0$ and
 $\mathcal A_{\widetilde\beta}^0$,
 respectively. In particular, both sums converge
 absolutely in their $\mathrm C^*$-algebras.
\end{lem}
\begin{proof}
 By \cite{Rie88}*{Corollary~2.4}, the
 Heisenberg matrix coefficients of $f,g\in\mathcal S(M)$,
 restricted to either lattice, are Schwartz functions on that
 lattice. The maps $\iota_\alpha$ and $\iota_\beta$ identify
 both lattices with $\Gamma_p$, whose Schwartz space is
 \[
 \mathcal S(\Gamma_p)
   =\bigcup_{N\geq0}\mathscr S(p^{-N}\Z^2).
 \]
 Indeed, every finitely generated subgroup of $\Gamma_p$
 is contained in some $p^{-N}\Z^2$.

 Multiplication by the scalar phases in our Weyl convention
 and changes of sign in the lattice indices preserve rapid
 decrease on each $p^{-N}\Z^2$. Hence the coefficient
 sequences belong to
 $\mathcal A_{\widetilde\alpha}^0$ and
 $\mathcal A_{\widetilde\beta}^0$, respectively.
 Rapidly decreasing sequences on $p^{-N}\Z^2$ are absolutely
 summable. Since $W_z$ and $M_z$ have norm one, both series
 converge absolutely in the $\mathrm C^*$-norm.
\end{proof}

\begin{thm}\label{thm:heisenberg-imprimitivity}
 With the module actions above and the inner products
 \eqref{eq:heisenberg-right-inner-product}--%
 \eqref{eq:heisenberg-left-inner-product},
 $\mathcal S(M)$ is a
 $\mathcal A_{\widetilde\beta}^0$--
 $\mathcal A_{\widetilde\alpha}^0$
 pre-imprimitivity bimodule. Its completion in
 \[
 \|f\|_{\mathcal E}
   =\|\langle f,f\rangle_{\A_{\widetilde\alpha}^\omega}\|^{1/2}
 \]
 is an $\A_{\widetilde\beta}^\omega$--
 $\A_{\widetilde\alpha}^\omega$
 imprimitivity bimodule, denoted by
 $\mathcal E_{\widetilde\beta,\widetilde\alpha}$.
 It is finitely generated and projective as a right
 $\A_{\widetilde\alpha}^\omega$-module.
\end{thm}
\begin{proof}
 We apply \cite{Rie88}*{Theorem~2.15} with $D=D_L$ and
 $D^\perp=D_R$. The lattice hypotheses follow from
 Lemma~\ref{lem:heisenberg-lattices}. To compare conventions,
 Rieffel uses
 $(\Pi(u,v)f)(s)=\chi(s,v)f(s+u)$, whereas
 \[
 \mathcal W(u,v)=\chi(-u/2,v)\Pi(-u,v).
 \]
 Thus our operators are obtained by changing scalar phases
 and reversing the translation coordinate, which preserves
 both lattices. Rieffel's right
 action uses adjoints; replacing the right lattice index by
 its negative gives our action $f\cdot W_z=R_zf$.
 Proposition~\ref{prop:heisenberg-module-actions} then
 identifies the coefficient multipliers as
 $\omega_{\widetilde\beta}$ on the left and
 $\omega_{\widetilde\alpha}$ on the right.

 Counting measure on $D_R$ corresponds, under Poisson duality,
 to the measure on $D_L$ with mass
 $\operatorname{covol}(D_R)^{-1}=(|x|_p\theta)^{-1}$ at each
 point. Expressing the left inner product using counting
 measure therefore gives the factor in
 \eqref{eq:heisenberg-left-inner-product}.
 Lemma~\ref{lem:heisenberg-coefficient-decay} shows that both
 inner products belong to the required dense algebras.

 Rieffel's theorem now gives positivity and definiteness
 of the inner products, fullness on both sides, and the
 compatibility identity
 \begin{equation}\label{eq:heisenberg-compatibility}
 {}_{\A_{\widetilde\beta}^\omega}\langle f,g\rangle\cdot h
   =f\cdot\langle g,h\rangle_{\A_{\widetilde\alpha}^\omega}
   \qquad(f,g,h\in\mathcal S(M)).
 \end{equation}
 These properties, together with the module identities and
 norm estimates in Rieffel's theorem, give the asserted
 imprimitivity bimodule after completion.
 Finite generation and projectivity follow from the lattice
 case in \cite{Rie88}*{Proposition~3.2}.
\end{proof}

Transporting the module operations and inner products through
$\Phi_{\widetilde\alpha}$ and $\Phi_{\widetilde\beta}$ from
\eqref{eq:Phi-standard-to-alternating} gives an
$\A_\beta$--$\A_\alpha$ imprimitivity bimodule. Hence we obtain
\[
 \A_\beta
 \ \cong\ \A_{\widetilde\beta}^{\omega}
 \ \sim_M\ \A_{\widetilde\alpha}^{\omega}
 \ \cong\ \A_\alpha.
\]
The following sections construct Weyl implementers and use
them to prove equivariance of this bimodule. The actions on
the standard solenoids are transported through these
isomorphisms as in \eqref{eq:transported-matrix-action}.

%% file: implementers_revised.tex
\section{Weyl operators and their matrix implementers}
\label{sec:weyl-metaplectic}

In this section, for each $A\in\mathrm{SL}_2(R_p)$, we construct a unitary
operator $S_A$ on $L^2(M,\mu)$ implementing the action of $A$
on the Weyl system introduced in Section~\ref{bimodule}.
These operators will be used to define cyclic actions on the
reciprocal bimodule in the following sections.

Since $R_p$ is a subring of both $\Q_p$ and $\R$, scalar
multiplication on $M$ is defined by
$r(q,t)=(rq,rt)$ for $r\in R_p$. For
$A=\begin{psmallmatrix}a&b\\c&d\end{psmallmatrix}
\in\mathrm{SL}_2(R_p)$, define its action on $G=M\times M$ by
\begin{equation}\label{eq:symplectic-action}
 A(u,v)=(au+bv,cu+dv),\qquad (u,v)\in G.
\end{equation}
Explicitly, if $u=(q,t)$ and $v=(q',t')$, then
\[
 A(u,v)=\bigl((aq+bq',at+bt'),(cq+dq',ct+dt')\bigr).
\]
Thus the same matrix acts on the pair $(q,q')\in\Q_p^2$
and on the pair $(t,t')\in\R^2$. These maps satisfy
$A(Bz)=(AB)z$ and $Iz=z$. Each is a continuous group
automorphism with inverse given by $A^{-1}$, so they define
an action of $\mathrm{SL}_2(R_p)$ on $G$.

\begin{lem}\label{lem:matrix-preserves-Weyl-multiplier}
 The action \eqref{eq:symplectic-action} preserves the Weyl
 multiplier $m$ and the commutator pairing $\sigma$:
 \[
 m(Az,Aw)=m(z,w),\qquad
 \sigma(Az,Aw)=\sigma(z,w)
 \]
 for all $A\in\mathrm{SL}_2(R_p)$ and $z,w\in G$.
\end{lem}
\begin{proof}
 Let $z=(u,v)$ and $w=(u',v')$. The pairing $\chi$ is symmetric
 and satisfies the identity \eqref{eq:heisenberg-balance}.
 Expanding
 \[
 m(Az,Aw)
  =\chi((au'+bv')/2,cu+dv)\,
     \overline{\chi((au+bv)/2,cu'+dv')}
 \]
 gives cancelling pairs of terms with coefficients $ac$ and
 $bd$. The remaining terms combine to
 \[
 \chi(u'/2,(ad-bc)v)\,
   \overline{\chi(u/2,(ad-bc)v')}.
 \]
 Since $ad-bc=1$, this equals $m(z,w)$. The identity
 $\sigma(z,w)=m(z,w)/m(w,z)$ then proves invariance of $\sigma$.
\end{proof}

\subsection{Irreducibility and scalar uniqueness}

For $u,v\in M$, let $\mathsf T_u$ and $\mathsf M_v$ be the translation
and modulation operators:
\[
 (\mathsf T_uf)(s)=f(s-u),\qquad
 (\mathsf M_vf)(s)=\chi(s,v)f(s).
\]
They are unitary on $L^2(M,\mu)$, preserve $\mathcal S(M)$, and
satisfy
\begin{equation}\label{eq:canonical-commutation}
 \mathsf M_v\mathsf T_u=\chi(u,v)\mathsf T_u\mathsf M_v.
\end{equation}
The Weyl normalization is
\begin{equation}\label{eq:weyl-definition}
 \mathcal W(u,v)
   =\chi(-u/2,v)\mathsf M_v\mathsf T_u.
\end{equation}
In particular, $\mathcal W(u,0)=\mathsf T_u$ and
$\mathcal W(0,v)=\mathsf M_v$.

\begin{lem}\label{lem:weyl-irreducible}
 Every bounded operator on $L^2(M,\mu)$ commuting with all
 $\mathcal W(u,v)$ is scalar. Consequently, two unitary
 operators implementing the same automorphism of the full
 Weyl system differ by a scalar of modulus one.
\end{lem}
\begin{proof}
 Since $\chi$ identifies $M$ with its Pontryagin dual, the
 translations $\mathsf T_u$ and modulations $\mathsf M_v$ give the
 canonical Heisenberg representation on $L^2(M,\mu)$.
 By \cite{Wei64}*{Theorem~1}, this representation is irreducible.
 The Weyl system has the same commutant as these translations
 and modulations, by \eqref{eq:weyl-definition}. Schur's lemma
 therefore shows that every bounded operator commuting with
 all Weyl operators is scalar.

 If $U$ and $V$ implement the same automorphism of the Weyl
 system, then $V^{-1}U$ commutes with every Weyl operator.
 Thus $V^{-1}U=\lambda I$, with $|\lambda|=1$ by unitarity.
\end{proof}

\subsection{Generators of the matrix group}

We use the matrices
\begin{equation}\label{eq:matrix-generators}
 J=\begin{pmatrix}0&1\\-1&0\end{pmatrix},\qquad
 L_r=\begin{pmatrix}1&0\\r&1\end{pmatrix}
       \ (r\in R_p),\qquad
 P=L_1,\qquad
 D=\begin{pmatrix}p&0\\0&p^{-1}\end{pmatrix}.
\end{equation}

\begin{lem}\label{lem:generators}
 The matrices $J,P,D$ generate $\mathrm{SL}_2(R_p)$.
\end{lem}
\begin{proof}
 See \cite{BM68}*{Section~4.1}.
\end{proof}

\subsection{Fourier transformation, shear, and dilation}

The generators are implemented by Fourier transformation,
multiplication by a quadratic phase, and dilation. On
$\mathcal S(M)$, define
\begin{align}
 (S_Jf)(s)
   &=\int_M\overline{\chi(h,s)}f(h)\,d\mu(h),
                       \label{eq:fourier-operator}\\
 (S_{L_r}f)(s)
   &=\chi(s,rs/2)f(s),
                       \label{eq:shear-operator}\\
 (S_Df)(s)
   &=f(s/p).
                       \label{eq:diagonal-operator}
\end{align}
We write $S_P=S_{L_1}$. These operators extend to unitaries on $L^2(M,\mu)$.
For $S_J$, this follows from Plancherel's theorem and the
self-duality of $\mu$. For $S_{L_r}$, it follows because
the quadratic phase has modulus one. Finally,
multiplication by $p$ on $M=\Q_p\times\R$ has modulus
\[
 |p|_p|p|_\infty=p^{-1}p=1.
\]
Thus $s\mapsto s/p$ preserves $\mu$, which explains the
absence of a normalizing scalar in \eqref{eq:diagonal-operator}.

The operators and their inverses act continuously on
$\mathcal S(M)$. Fourier transformation preserves the
$p$-adic Bruhat space and the real Schwartz space.
For the shear, the $p$-adic quadratic phase is locally
constant on any compact support, while derivatives of
the real phase grow at most polynomially. Dilation
preserves both spaces. The inverse operators are
\[
 (S_J^{-1}f)(s)=\int_M\chi(h,s)f(h)\,d\mu(h),\qquad
 S_{L_r}^{-1}=S_{L_{-r}},\qquad
 (S_D^{-1}f)(s)=f(ps).
\]

\begin{prp}\label{prop:fourier-covariance}
 For $u,v\in M$ and $r\in R_p$, one has
 \begin{align}
 S_J\mathcal W(u,v)S_J^{-1}
   &=\mathcal W(v,-u),\label{eq:fourier-covariance}\\
 S_{L_r}\mathcal W(u,v)S_{L_r}^{-1}
   &=\mathcal W(u,v+ru),\label{eq:shear-covariance}\\
 S_D\mathcal W(u,v)S_D^{-1}
   &=\mathcal W(pu,p^{-1}v).
                       \label{eq:dilation-covariance}
 \end{align}
\end{prp}
\begin{proof}
 We first work on $\mathcal S(M)$. A change of variable
 in the Fourier integral gives
 \[
 S_J\mathsf T_u=\mathsf M_{-u}S_J,\qquad
 S_J\mathsf M_v=\mathsf T_vS_J.
 \]
 Using \eqref{eq:weyl-definition} and
 \eqref{eq:canonical-commutation}, we obtain
 \[
 \begin{aligned}
 S_J\mathcal W(u,v)S_J^{-1}
   &=\chi(-u/2,v)\mathsf T_v\mathsf M_{-u}\\
   &=\chi(-u/2,v)\chi(v,u)\mathsf M_{-u}\mathsf T_v\\
   &=\chi(v/2,u)\mathsf M_{-u}\mathsf T_v
     =\mathcal W(v,-u).
 \end{aligned}
 \]

 For the shear, put $C_r(s)=\chi(s,rs/2)$. Symmetry
 and balance of $\chi$ give
 \[
 C_r(s)\overline{C_r(s-u)}
    =\chi(s,ru)\chi(-u/2,ru).
 \]
 Therefore
 \[
 S_{L_r}\mathsf T_uS_{L_r}^{-1}
    =\chi(-u/2,ru)\mathsf M_{ru}\mathsf T_u.
 \]
 Since $S_{L_r}$ commutes with each modulation, it follows that
 \[
 \begin{aligned}
 S_{L_r}\mathcal W(u,v)S_{L_r}^{-1}
   &=\chi(-u/2,v+ru)\mathsf M_{v+ru}\mathsf T_u\\
   &=\mathcal W(u,v+ru).
 \end{aligned}
 \]

 For dilation, direct substitution gives
 \[
 S_D\mathsf T_uS_D^{-1}=\mathsf T_{pu},\qquad
 S_D\mathsf M_vS_D^{-1}=\mathsf M_{p^{-1}v}.
 \]
 Since $\chi(-pu/2,p^{-1}v)=\chi(-u/2,v)$, this proves
 \eqref{eq:dilation-covariance}. All three identities extend
 to $L^2(M,\mu)$ by density of $\mathcal S(M)$ and boundedness
 of the operators.
\end{proof}

\subsection{The projective representation}

Let $\mathcal H=L^2(M,\mu)$ and write
$PU(\mathcal H)=U(\mathcal H)/(\T I)$ for the projective
unitary group. For a unitary $U$, denote its class by $[U]$.
Scalar uniqueness allows the generator implementers to
define a representation in this quotient. This is the
projective form of the Weil representation \cite{Wei64}.

\begin{thm}\label{thm:projective-metaplectic}
 For every $A\in\mathrm{SL}_2(R_p)$, there is a unitary
 $S_A$ on $\mathcal H$, preserving $\mathcal S(M)$, such that
 \begin{equation}\label{eq:general-covariance}
 S_A\mathcal W(z)S_A^{-1}=\mathcal W(Az)
        \qquad(z\in G).
 \end{equation}
 Its projective class $[S_A]$ is uniquely determined
 by this covariance. The assignment
 \[
 \mathrm{SL}_2(R_p)\longrightarrow PU(\mathcal H),
       \qquad A\longmapsto[S_A],
 \]
 is a group homomorphism.
\end{thm}
\begin{proof}
 By Lemma~\ref{lem:generators}, write $A=A_1\cdots A_k$,
 where each $A_j$ is $J,P,D$ or one of their inverses, and set
 $S_A=S_{A_1}\cdots S_{A_k}$. Repeated application of
 Proposition~\ref{prop:fourier-covariance} gives
 \eqref{eq:general-covariance}. Both $S_A$ and its inverse
 preserve $\mathcal S(M)$.

 Two words representing $A$ give unitaries with the same
 conjugation action on every Weyl operator. By
 Lemma~\ref{lem:weyl-irreducible}, these unitaries differ by
 a scalar of modulus one. Thus $[S_A]$ is independent of the
 chosen word and, more generally, of the choice of implementer.

 Similarly, $S_AS_B$ and $S_{AB}$ both implement the action
 of $AB$, so
 \begin{equation}\label{eq:projective-multiplier}
 S_AS_B=c(A,B)S_{AB},\qquad c(A,B)\in\T.
 \end{equation}
 Hence $[S_A][S_B]=[S_{AB}]$. Choosing $S_I=I$, associativity
 gives
 \[
 c(A,B)c(AB,C)=c(B,C)c(A,BC),\qquad
 c(I,A)=c(A,I)=1.
 \]
 Thus $c$ is a normalized $2$-cocycle, and the assignment of
 projective classes is a group homomorphism.
\end{proof}

For a fixed $A$, the powers $S_A^m$, $m\in\Z$, form a unitary
representation of $\Z$ implementing $A^m$. If $A$ has finite
order $n$, covariance and Lemma~\ref{lem:weyl-irreducible} give
$S_A^n=\lambda I$ for some $\lambda\in\T$. Choose $\zeta\in\T$
with $\zeta^n=\lambda^{-1}$. Then $\zeta S_A$ still implements
$A$ and satisfies $(\zeta S_A)^n=I$, so its powers define a
unitary representation of $\Z/n\Z$. These representations
will be used for the bimodule actions in the following sections.

%% file: equivariance_revised.tex
\section{Equivariance and crossed products by \texorpdfstring{$\Z$}{Z}}
\label{Morita eq}

We use the Weyl implementers of Section~\ref{sec:weyl-metaplectic}
to define actions on the reciprocal bimodule of
Section~\ref{bimodule}. We first determine the matrices induced
on its two coefficient algebras, then verify covariance of the
module operations and inner products. This gives Morita
equivalence of the corresponding crossed products by $\Z$.

\subsection{Equivariant Morita equivalence}

We recall the equivariant imprimitivity theorem for discrete
groups. Crossed products are reduced crossed products; for the
cyclic groups considered here, the full and reduced versions
coincide because these groups are amenable.

\begin{thm}[Combes; Curto--Muhly--Williams]
\label{M.E of tori}
 Let $B$ and $A$ be $\mathrm C^*$-algebras with actions
 $\gamma$ and $\delta$ of a discrete group $H$, and let
 $E$ be a $B$--$A$ imprimitivity bimodule. Suppose that $H$
 acts on $E$ by complex-linear bijections $\tau_h$ such that,
 for $b\in B$, $a\in A$, and $f,g\in E$,
 \begin{align*}
 \tau_h(bfa)
   &=\gamma_h(b)\tau_h(f)\delta_h(a),\\
 \langle\tau_h(f),\tau_h(g)\rangle_A
   &=\delta_h(\langle f,g\rangle_A),\\
 {}_B\langle\tau_h(f),\tau_h(g)\rangle
   &=\gamma_h({}_B\langle f,g\rangle).
 \end{align*}
 Then $B\rtimes_\gamma H$ and $A\rtimes_\delta H$ are
 strongly Morita equivalent.
\end{thm}

For equivariant imprimitivity, see
\cite{Curto84}*{Theorem~1} and \cite{Combes84}.
The reduced statement follows from the linking-algebra argument:
the identities define an action on
\[
 \begin{pmatrix}B&E\\E^*&A\end{pmatrix}.
\]
In the regular representation, its diagonal corners identify with
$B\rtimes_\gamma H$ and $A\rtimes_\delta H$.
The invariant corner projections remain full in the reduced
crossed product. The
off-diagonal corner is their imprimitivity bimodule, denoted
$E\rtimes H$. Since $H$ is discrete, continuity of the
action on $E$ is automatic.

\subsection{The matrices induced on the two lattices}

Put
\[
 T=\begin{pmatrix}1&0\\0&-1\end{pmatrix}.
\]
The sign in the left embedding
$\iota_\beta(a,b)=(\ell(a),-\ell(b))$ determines the
matrix acting on the left coefficient algebra.

\begin{lem}\label{lem:two-lattice-matrices}
 For $A\in\mathrm{SL}_2(R_p)$, put $B=TAT$.
 Then, for every $z\in\Gamma_p$,
 \begin{equation}\label{eq:two-lattice-matrices}
 A\iota_\alpha(z)=\iota_\alpha(Az),
 \qquad
 A\iota_\beta(z)=\iota_\beta(Bz).
 \end{equation}
\end{lem}
\begin{proof}
 Write $A=\begin{psmallmatrix}r&s\\t&u\end{psmallmatrix}$
 and $z=(a,b)$. The $R_p$-linearity of $\rho$ gives
 \[
 A\iota_\alpha(a,b)
  =(-\rho(ra+sb),-\rho(ta+ub))
  =\iota_\alpha(Az).
 \]
 Similarly, the $R_p$-linearity of $\ell$ gives
 \[
 \begin{aligned}
 A\iota_\beta(a,b)
  &=(\ell(ra-sb),\ell(ta-ub))=\iota_\beta(ra-sb,-ta+ub)
   =\iota_\beta(Bz),
 \end{aligned}
 \]
 since $B=\begin{psmallmatrix}r&-s\\-t&u\end{psmallmatrix}$.
\end{proof}

Note that both $A$ and $B$ belong to $\mathrm{SL}_2(R_p)$. We use
the actions $\lambda_A^{\widetilde\alpha}$ and
$\lambda_B^{\widetilde\beta}$ on the alternating-multiplier
algebras associated with the compatible even lifts of
Section~\ref{bimodule}.

\begin{prp}\label{prop:heisenberg-equivariance}
 Let $A\in\mathrm{SL}_2(R_p)$, put $B=TAT$, and let
 $S_A$ be a unitary implementer from
 Theorem~\ref{thm:projective-metaplectic}. For
 $f,g\in\mathcal S(M)$,
 $a\in\mathcal A_{\widetilde\alpha}^0$, and
 $b\in\mathcal A_{\widetilde\beta}^0$, one has
 \begin{align}
 S_A(f\cdot a)
   &=S_Af\cdot\lambda_A^{\widetilde\alpha}(a),
                \label{eq:equivariant-right-action}\\
 S_A(b\cdot f)
   &=\lambda_B^{\widetilde\beta}(b)\cdot S_Af,
                \label{eq:equivariant-left-action}\\
 \langle S_Af,S_Ag\rangle_{\A_{\widetilde\alpha}^\omega}
   &=\lambda_A^{\widetilde\alpha}
       (\langle f,g\rangle_{\A_{\widetilde\alpha}^\omega}),
                \label{eq:equivariant-right-inner-product}\\
 {}_{\A_{\widetilde\beta}^\omega}\langle S_Af,S_Ag\rangle
   &=\lambda_B^{\widetilde\beta}
       ({}_{\A_{\widetilde\beta}^\omega}\langle f,g\rangle).
                \label{eq:equivariant-left-inner-product}
 \end{align}
 The operator $S_A$ extends to a bijective isometry of
 $\mathcal E_{\widetilde\beta,\widetilde\alpha}$.
 All four identities remain valid on the completed
 bimodule and coefficient algebras.
\end{prp}
\begin{proof}
 Weyl covariance \eqref{eq:general-covariance} and
 Lemma~\ref{lem:two-lattice-matrices} give
 \begin{equation}\label{eq:equivariant-lattice-operators}
 S_AR_zS_A^{-1}=R_{Az},
 \qquad
 S_AL_zS_A^{-1}=L_{Bz}.
 \end{equation}
 Hence, on the canonical generators,
 \[
 \begin{aligned}
 S_A(f\cdot W_z)
   &=S_AR_zf
    =R_{Az}S_Af
    =S_Af\cdot W_{Az},\\
 S_A(M_z\cdot f)
   &=S_AL_zf
    =L_{Bz}S_Af
    =M_{Bz}\cdot S_Af.
 \end{aligned}
 \]
 These are the two module identities for finite sums.

 The coefficient actions preserve
 $\mathcal A_{\widetilde\alpha}^0$ and
 $\mathcal A_{\widetilde\beta}^0$. Indeed, choose $k\geq0$
 so that the entries of $A^{\pm1}$ and $B^{\pm1}$ lie in
 $p^{-k}\Z$. Each matrix sends $p^{-N}\Z^2$ into
 $p^{-(N+k)}\Z^2$, and invertibility over $\R$ preserves
 rapid decay. The series defining the module operations
 converge in $\mathcal S(M)$, as established in
 Section~\ref{bimodule}. Since $S_A$ is continuous on
 $\mathcal S(M)$, the two identities hold for these
 coefficient algebras as well.

 For the right inner product, set $v=Az$ in its defining
 series. Covariance and unitarity give
 \[
 \begin{aligned}
 \langle S_Af,S_Ag\rangle_{\A_{\widetilde\alpha}^\omega}
   &=\sum_{z\in\Gamma_p}
       (R_{-Az}S_Ag,S_Af)_2\,W_{Az}\\
   &=\sum_{z\in\Gamma_p}(R_{-z}g,f)_2\,W_{Az}\\
   &=\lambda_A^{\widetilde\alpha}
       (\langle f,g\rangle_{\A_{\widetilde\alpha}^\omega}).
 \end{aligned}
 \]
 For the left inner product, the substitution $v=Bz$
 gives, with $\kappa=|x|_p\theta$,
 \[
 \begin{aligned}
 {}_{\A_{\widetilde\beta}^\omega}\langle S_Af,S_Ag\rangle
   &=\kappa^{-1}\sum_{z\in\Gamma_p}
       (S_Af,L_{Bz}S_Ag)_2\,M_{Bz}\\
   &=\kappa^{-1}\sum_{z\in\Gamma_p}(f,L_zg)_2\,M_{Bz}\\
   &=\lambda_B^{\widetilde\beta}
       ({}_{\A_{\widetilde\beta}^\omega}\langle f,g\rangle).
 \end{aligned}
 \]
 Both series converge absolutely by
 Lemma~\ref{lem:heisenberg-coefficient-decay}, so the
 reindexings are valid.

 The right inner-product identity implies
 \[
 \|S_Af\|_{\mathcal E}^2
   =\|\lambda_A^{\widetilde\alpha}
       (\langle f,f\rangle_{\A_{\widetilde\alpha}^\omega})\|
   =\|f\|_{\mathcal E}^2.
 \]
 Thus $S_A$ extends isometrically to the completion.
 Its inverse preserves $\mathcal S(M)$ and satisfies
 the inverse covariance identities. The same argument
 extends it isometrically, and the two extensions are
 inverse on the dense subspace $\mathcal S(M)$, hence on
 the completion. Finally, continuity of the module
 operations and inner products extends all four
 identities to the completed objects.
\end{proof}

\subsection{The \texorpdfstring{$\Z$}{Z}-action and examples}

For a fixed matrix $A$, powers of one implementer define
a $\Z$-action on the bimodule.

\begin{thm}\label{Z action Morita}
 Let $p$ be an odd prime and let $\alpha\in\Xi_p$ satisfy
 $\theta=\alpha_0>0$ and $x=x_\alpha\ne0$. Let
 $\beta\in\Omega_p$ be the reciprocal parameter in
 \eqref{eq:reciprocal-parameters}, and choose the compatible
 even lifts $\widetilde\alpha,\widetilde\beta$ of
 Section~\ref{bimodule}. For every $A\in\mathrm{SL}_2(R_p)$,
 put $B=TAT$. The bimodule
 $\mathcal E_{\widetilde\beta,\widetilde\alpha}$ admits a
 $\Z$-action compatible with the coefficient actions
 generated by $\lambda_B^{\widetilde\beta}$ on the left
 and $\lambda_A^{\widetilde\alpha}$ on the right.
 Consequently,
 \[
 \A_{\widetilde\beta}^\omega
     \rtimes_{\lambda_B^{\widetilde\beta}}\Z
   \ \sim_M\
 \A_{\widetilde\alpha}^\omega
     \rtimes_{\lambda_A^{\widetilde\alpha}}\Z .
 \]
\end{thm}
\begin{proof}
 By Proposition~\ref{prop:heisenberg-equivariance}, $S_A$
 extends to a bijective isometry of the imprimitivity
 bimodule. Define $\tau_m=S_A^m$ for $m\in\Z$.
 Then $\tau_{m+k}=\tau_m\tau_k$ and $\tau_0=I$.
 Iterating \eqref{eq:equivariant-lattice-operators} gives
 \[
 S_A^mR_zS_A^{-m}=R_{A^mz},
 \qquad
 S_A^mL_zS_A^{-m}=L_{B^mz},
 \]
 where $B^m=TA^mT$. The four covariance identities in
 Proposition~\ref{prop:heisenberg-equivariance} therefore
 hold for $\tau_m$, with coefficient actions
 $\lambda_{A^m}^{\widetilde\alpha}$ and
 $\lambda_{B^m}^{\widetilde\beta}$.
 Theorem~\ref{M.E of tori} gives the asserted Morita
 equivalence and proves Theorem~\ref{intro theorem Z}.
\end{proof}

The induced bimodule
$\mathcal E_{\widetilde\beta,\widetilde\alpha}\rtimes\Z$
is the completion of the finitely supported functions
$\Z\to\mathcal E_{\widetilde\beta,\widetilde\alpha}$.
Since both crossed products are unital,
\cite{Rie88}*{Proposition~1.2}, applied to the identity
projection of the left crossed product, shows that it is
finitely generated projective as a right
$\A_{\widetilde\alpha}^\omega
\rtimes_{\lambda_A^{\widetilde\alpha}}\Z$-module.

The theorem requires neither $x$ to be a $p$-adic unit nor
the transition integers of the original parameters to be
even. The compatible lifts are used to construct the
bimodule. For any other even lifts of the same parameters,
the equivariant isomorphisms in
Corollary~\ref{cor:lift-independent-crossed-products}
transport the module operations and inner products, giving
the same crossed-product Morita equivalence.

\begin{proof}[Proof of Corollary~\ref{intro corollary standard cyclic} for $\Z$]
The equivariant isomorphisms $\Phi_{\widetilde\alpha}$ and
$\Phi_{\widetilde\beta}$ transport the equivalence in
Theorem~\ref{Z action Morita} to the standard solenoids:
\[
 \A_\beta\rtimes_{\delta_B^{\widetilde\beta}}\Z
  \ \sim_M\
 \A_\alpha\rtimes_{\delta_A^{\widetilde\alpha}}\Z,
\]
where the actions are defined by
\eqref{eq:transported-matrix-action}. These actions may
include phases on the standard generators.
By Corollary~\ref{cor:lift-independent-crossed-products},
replacing either lift gives an isomorphic crossed product.
Hence the equivalence holds for any of the representatives
$\delta^\alpha$ and $\delta^\beta$ used in the introduction.
\end{proof}

For every $\zeta\in\T$, the powers of $\zeta S_A$ give a
compatible $\Z$-action. If $A$ has finite order $n$, we must choose the
scalar so that the resulting implementer has $n$-th power
$I$; this is done in the next section. For a general
subgroup $H\subseteq\mathrm{SL}_2(R_p)$, the construction would
require compatible scalar choices satisfying
$S_hS_k=S_{hk}$ for all $h,k\in H$.

\begin{exa}\label{ex:all-odd-primes}
 Let $p$ be an odd prime and $0<\theta<1$. Take $x=y=-1$,
 whose $p$-adic expansion is
 \[
 -1=\sum_{j\geq0}(p-1)p^j.
 \]
 Since $p-1$ is even, the compatible even lifts are
 \[
 \alpha_n=1+\frac{\theta-1}{p^n},
 \qquad
 \beta_n=1+\frac{\theta^{-1}-1}{p^n}.
 \]
 Here $\alpha\in\Xi_p$, and $\beta$ is its reciprocal
 parameter. The theorem applies for every odd prime and
 for both rational and irrational $\theta$.
\end{exa}

\begin{exa}\label{ex:dilation-and-shear-actions}
 Under the hypotheses of Theorem~\ref{Z action Morita},
 take $A=D=\operatorname{diag}(p,p^{-1})$.
 Then $TDT=D$, and the compatible module operator is
 \[
 (S_Df)(q,t)=f(q/p,t/p).
 \]
 Thus the two coefficient actions use the same matrix.
 For $A=P=L_1$, one has $TPT=P^{-1}$. The theorem then
 relates the $P^{-1}$-action on the left solenoid to the
 $P$-action on the right, with module operator
 \[
 (S_Pf)(s)=\chi(s,s/2)f(s).
 \]
\end{exa}

%% file: finite_cyclic_revised.tex
\section{Finite cyclic groups and their crossed products}
\label{Morita eq 2}

For a matrix of finite order, a scalar normalization of its
Weyl implementer gives an action of the corresponding cyclic
group on the reciprocal bimodule. We use this action to prove
crossed-product Morita equivalence at odd primes. A separate
argument for the flip works for every prime. We then describe
the finite subgroups of $\mathrm{SL}_2(R_p)$ and determine
when their integral and $p$-twisted representatives are
conjugate.

\subsection{Normalization of a finite-order implementer}

If $A$ has order $n$, its implementer $S_A$ may satisfy
$S_A^n=\lambda I$ with $\lambda\ne1$. Multiplying $S_A$
by a suitable scalar removes this obstruction.

\begin{lem}\label{lem:finite-implementer-normalization}
 Let $A\in\mathrm{SL}_2(R_p)$ have exact order $n$, and let
 $S_A$ be a unitary implementer from
 Theorem~\ref{thm:projective-metaplectic}. There is
 $\zeta\in\T$ such that $\widetilde S_A=\zeta S_A$
 satisfies
 \[
 \widetilde S_A^{\,n}=I,\qquad
 \widetilde S_A\mathcal W(z)\widetilde S_A^{-1}
       =\mathcal W(Az)\quad(z\in G).
 \]
 Both $\widetilde S_A$ and its inverse preserve
 $\mathcal S(M)$.
\end{lem}
\begin{proof}
 Since $A^n=I$, Weyl covariance gives
 \[
 S_A^n\mathcal W(z)S_A^{-n}=\mathcal W(z)
       \qquad(z\in G).
 \]
 Thus $S_A^n$ commutes with every Weyl operator.
 Lemma~\ref{lem:weyl-irreducible} implies
 $S_A^n=\lambda I$, with $\lambda\in\T$ because $S_A$
 is unitary. Choose $\zeta\in\T$ with
 $\zeta^n=\lambda^{-1}$. Then $(\zeta S_A)^n=I$.
 Scalar multiplication preserves covariance and the
 invariance of $\mathcal S(M)$.
\end{proof}

\begin{thm}\label{finite Morita}
 Let $p$ be an odd prime and let $\alpha\in\Xi_p$ satisfy
 $\theta=\alpha_0>0$ and $x_\alpha\ne0$. Let
 $\beta\in\Omega_p$ be the reciprocal parameter in
 \eqref{eq:reciprocal-parameters}, and choose the compatible
 even lifts $\widetilde\alpha,\widetilde\beta$ of
 Section~\ref{bimodule}. If $A\in\mathrm{SL}_2(R_p)$
 has exact order $n$ and $B=TAT$, then
 $\mathcal E_{\widetilde\beta,\widetilde\alpha}$ admits
 a $\Z/n\Z$-action compatible with
 $\bar k\mapsto\lambda_{B^k}^{\widetilde\beta}$ on the
 left algebra and
 $\bar k\mapsto\lambda_{A^k}^{\widetilde\alpha}$ on the
 right algebra. Consequently,
 \[
 \A_{\widetilde\beta}^\omega
      \rtimes_{\lambda_B^{\widetilde\beta}}(\Z/n\Z)
       \ \sim_M\
 \A_{\widetilde\alpha}^\omega
      \rtimes_{\lambda_A^{\widetilde\alpha}}(\Z/n\Z).
 \]
\end{thm}
\begin{proof}
 Choose $\widetilde S_A=\zeta S_A$ as in
 Lemma~\ref{lem:finite-implementer-normalization}.
 The four identities of
 Proposition~\ref{prop:heisenberg-equivariance} hold
 for $\widetilde S_A$: scalar multiplication preserves
 module covariance, and the scalar factors in each
 inner product contribute $|\zeta|^2=1$.
 Thus $\widetilde S_A$ extends to a bijective isometry
 of the completed imprimitivity bimodule.

 Define
 \[
 \tau_{\bar k}=\widetilde S_A^{\,k},
       \qquad \bar k\in\Z/n\Z.
 \]
 The identity $\widetilde S_A^{\,n}=I$ makes this
 well-defined and gives the group law. Taking powers
 of the covariance identities gives the coefficient
 actions $\lambda_{A^k}^{\widetilde\alpha}$ and
 $\lambda_{B^k}^{\widetilde\beta}$, since $B^k=TA^kT$.
 Theorem~\ref{M.E of tori} now gives the asserted
 equivalence and proves Theorem~\ref{intro theorem finite}.
\end{proof}

As in Section~\ref{Morita eq}, the induced bimodule
$\mathcal E_{\widetilde\beta,\widetilde\alpha}
\rtimes(\Z/n\Z)$ is finitely generated projective as a right
$\A_{\widetilde\alpha}^\omega
\rtimes_{\lambda_A^{\widetilde\alpha}}(\Z/n\Z)$-module,
since both crossed products are unital.

\begin{proof}[Proof of Corollary~\ref{intro corollary standard cyclic} for $\Z/n\Z$]
The equivariant isomorphisms $\Phi_{\widetilde\alpha}$ and
$\Phi_{\widetilde\beta}$ transport the equivalence in
Theorem~\ref{finite Morita} to the standard solenoids:
\[
 \A_\beta\rtimes_{\delta_B^{\widetilde\beta}}(\Z/n\Z)
   \ \sim_M\
 \A_\alpha\rtimes_{\delta_A^{\widetilde\alpha}}(\Z/n\Z).
\]
The actions are those of \eqref{eq:transported-matrix-action}
and may include phases on standard generators. By
Corollary~\ref{cor:lift-independent-crossed-products},
the crossed products are independent, up to isomorphism,
of the chosen even lifts.
Thus the equivalence holds for any of the representatives
$\delta^\alpha$ and $\delta^\beta$ used in the introduction.
\end{proof}

The theorem applies to every finite-order matrix, including
the identity. The possible nontrivial orders are $2,3,4,6$,
as shown below.

\subsection{The flip for Lu's reciprocal parameters}

By Proposition~\ref{prop:standard-flip-action}, $-I$ acts on
every standard solenoid by $W_z^\Psi\mapsto W_{-z}^\Psi$.
For odd primes, the flip equivalence follows from
Theorem~\ref{finite Morita} and the identification in
Corollary~\ref{cor:standard-alternating-flip}. Reflection
on the standard Heisenberg bimodule also gives the result
at $p=2$.

\begin{thm}\label{thm:lu-flip-morita}
 Let $p$ be any prime and let $\alpha\in\Xi_p$ satisfy
 $\theta=\alpha_0>0$ and the associated $p$-adic integer $x=x_\alpha=\sum_{j\geq0}x_jp^j\ne0.$
 Put $v=v_p(x)$ and write
 $y=x^{-1}=\sum_{j=-v}^{\infty}y_jp^j\in\Q_p$.
 Define the reciprocal parameter $\beta\in\Omega_p$ by
 \begin{equation}\label{eq:lu-general-reciprocal}
  \beta_n=\frac{\theta^{-1}
                    +\sum_{j=-v}^{n-1}y_jp^j}{p^n}
        \qquad(n\geq0).
 \end{equation}
 Then
 \[
 \A_\beta\rtimes_{\mathrm{flip}_\beta}(\Z/2\Z)
     \ \sim_M\
 \A_\alpha\rtimes_{\mathrm{flip}_\alpha}(\Z/2\Z).
 \]
\end{thm}
\begin{proof}
 For odd $p$, take $A=-I$ in Theorem~\ref{finite Morita}.
 Since $T(-I)T=-I$, both coefficient actions are flips.
 Corollary~\ref{cor:standard-alternating-flip} identifies
 them with the standard flips after transport through
 $\Phi_{\widetilde\alpha}$ and $\Phi_{\widetilde\beta}$.
 Thus the bimodule already constructed gives the result.

 For $p=2$, use the standard Heisenberg
 $\A_\beta$--$\A_\alpha$ bimodule of
 \cite{Lu22}*{Theorem~3.6}, whose dense subspace is
 $\mathcal S(M)$. The formulas for $M$, $\chi$, $\mu$,
 $\rho$, and $\ell$ in Section~\ref{bimodule} also apply
 at $p=2$. Denote the canonical unitaries of the standard
 solenoids $\A_\alpha$ and $\A_\beta$ by
 $W_z^{\Psi_\alpha}$ and $W_z^{\Psi_\beta}$, respectively.
 Their right and left actions are given, for
 $z=(a,b)\in\Gamma_p$, by the operators
 \begin{equation}\label{eq:lu-standard-module-actions}
 R_z^\Psi=\mathsf T_{-\rho(a)}\mathsf M_{-\rho(b)},
 \qquad
 L_z^\Psi=\mathsf M_{-\ell(b)}\mathsf T_{\ell(a)},
 \end{equation}
 where $(\mathsf T_u f)(s)=f(s-u)$ and
 $(\mathsf M_vf)(s)=\chi(s,v)f(s)$.
 The inner products are the lattice sums
 \begin{align*}
 \langle f,g\rangle_{\A_\alpha}
   &=\sum_z((R_z^\Psi)^*g,f)_2W_z^{\Psi_\alpha},\\
 {}_{\A_\beta}\langle f,g\rangle
   &=(|x|_p\theta)^{-1}\sum_z(f,L_z^\Psi g)_2W_z^{\Psi_\beta}.
 \end{align*}
 Both sums converge absolutely by the coefficient-decay
 argument of Lemma~\ref{lem:heisenberg-coefficient-decay},
 which also applies at $p=2$.

 Define $(\mathcal Rf)(s)=f(-s)$. Reflection is unitary,
 preserves $\mathcal S(M)$, and satisfies
 \[
 \mathcal R^2=I,\qquad
 \mathcal R R_z^\Psi\mathcal R=R_{-z}^\Psi,\qquad
 \mathcal R L_z^\Psi\mathcal R=L_{-z}^\Psi.
 \]
 Hence it intertwines both module actions with the flips.
 Unitarity and reindexing $z\mapsto-z$ in the
 inner-product sums give
 \[
 \begin{aligned}
 \langle\mathcal Rf,\mathcal Rg\rangle_{\A_\alpha}
   &=\mathrm{flip}_\alpha(\langle f,g\rangle_{\A_\alpha}),\\
 {}_{\A_\beta}\langle\mathcal Rf,\mathcal Rg\rangle
   &=\mathrm{flip}_\beta({}_{\A_\beta}\langle f,g\rangle).
 \end{aligned}
 \]
 Thus $\mathcal R$ extends to an involutive isometry
 of the completed bimodule. By continuity, the module
 and inner-product identities hold on the completion,
 giving a compatible $\Z/2\Z$-action.
 Theorem~\ref{M.E of tori} proves the asserted equivalence.
\end{proof}

The induced crossed-product bimodule is finitely generated
projective as a right
$\A_\alpha\rtimes_{\mathrm{flip}_\alpha}(\Z/2\Z)$-module.
The hypotheses $\alpha_0>0$ and $x_\alpha\ne0$ are needed
for this reciprocal construction. The standard flip itself
exists for every solenoid parameter.

\subsection{Finite subgroups and representatives}

For the remainder of this section, $p$ is prime.

\begin{prp}\label{prop:finite-orders}
 Every finite subgroup of $\mathrm{SL}_2(R_p)$ is cyclic.
 Every finite-order element has order $1,2,3,4$, or $6$.
\end{prp}
\begin{proof}
 Let $H\leq\mathrm{SL}_2(R_p)$ be finite. Averaging the
 Euclidean inner product on $\R^2$ over $H$ gives an
 $H$-invariant positive definite inner product. In an
 orthonormal basis, $H$ is a subgroup of $\mathrm{SO}(2)$.
 Finite subgroups of the circle group are cyclic.

 For the possible orders, we use the trace argument of
 \cite{Con}*{Theorem~2.7}, which applies since $R_p\subset\Q$.
 If $A$ has finite order, its minimal polynomial divides
 $X^n-1$ for some $n$. Thus $A$ is diagonalizable over
 $\C$, with eigenvalues $\lambda,\lambda^{-1}$ that
 are roots of unity. Its trace is both rational and an
 algebraic integer, hence an integer. Since
 $\operatorname{tr}(A)=2\operatorname{Re}\lambda$,
 \[
 \operatorname{tr}(A)\in\{-2,-1,0,1,2\}.
 \]
 For trace $2$ or $-2$, diagonalizability gives $A=I$
 or $A=-I$. For traces $-1,0,1$, the characteristic
 polynomials are $X^2+X+1$, $X^2+1$, and $X^2-X+1$,
 whose roots have exact orders $3,4,6$, respectively.
\end{proof}

Use the integral representatives
\begin{equation}\label{eq:finite-integral-representatives}
 W_2=-I,\quad
 W_3=\begin{pmatrix}-1&1\\-1&0\end{pmatrix},
 \quad
 W_4=J=\begin{pmatrix}0&1\\-1&0\end{pmatrix},
 \quad
 W_6=\begin{pmatrix}0&1\\-1&1\end{pmatrix}.
\end{equation}
Put
\[
 D_p=\operatorname{diag}(p,1)\in\mathrm{GL}_2(R_p),
 \qquad W_{i,p}=D_pW_iD_p^{-1}
             \quad(i=3,4,6).
\]
The matrix $D_p$ has determinant $p$ and is used to
describe the second copy of $\mathrm{SL}_2(\Z)$ in
the amalgam below. Explicitly,
\begin{equation}\label{eq:finite-twisted-representatives}
 W_{3,p}=
  \begin{pmatrix}-1&p\\-p^{-1}&0\end{pmatrix},
 \quad
 W_{4,p}=
  \begin{pmatrix}0&p\\-p^{-1}&0\end{pmatrix},
 \quad
 W_{6,p}=
  \begin{pmatrix}0&p\\-p^{-1}&1\end{pmatrix}.
\end{equation}

\begin{thm}\label{thm:finite-subgroup-representatives}
 Every nontrivial finite subgroup of $\mathrm{SL}_2(R_p)$
 is conjugate within $\mathrm{SL}_2(R_p)$ to one of
 \[
 \langle W_2\rangle,\quad
 \langle W_3\rangle,\quad
 \langle W_{3,p}\rangle,\quad
 \langle W_4\rangle,\quad
 \langle W_{4,p}\rangle,\quad
 \langle W_6\rangle,\quad
 \langle W_{6,p}\rangle.
 \]
 Some groups in this list may be conjugate to one another.
\end{thm}
\begin{proof}
 The Ihara--Serre decomposition
 \cite{Ser80}*{Chapter~II, Section~1.4, Corollary~2}
 gives
 \begin{equation}\label{eq:Ihara-Serre-amalgam}
 \mathrm{SL}_2(R_p)
   =G_1*_{G_1\cap G_2}G_2,\quad
 G_1=\mathrm{SL}_2(\Z),\quad
 G_2=D_p\mathrm{SL}_2(\Z)D_p^{-1}.
 \end{equation}
 With this choice of $D_p$, the intersection consists of
 integral determinant-one matrices whose upper-right
 entry is divisible by $p$: for an integral matrix $X$,
 \[
 D_p^{-1}XD_p
   =\begin{pmatrix}a&b/p\\pc&d\end{pmatrix}
       \quad\text{if }X=\begin{pmatrix}a&b\\c&d\end{pmatrix}.
 \]

 A finite group acting without inversions on a tree fixes
 a vertex \cite{Ser80}*{p.~36}. Indeed, the convex hull
 of a finite orbit is a finite invariant tree. Its center
 is a vertex or an edge, and in the latter case both
 endpoints are fixed because there are no inversions.
 Applied to the tree of \eqref{eq:Ihara-Serre-amalgam},
 this shows that every finite subgroup is conjugate
 into $G_1$ or $G_2$.

 The standard decomposition
 \[
 \mathrm{SL}_2(\Z)
   \cong\Z/4\Z*_{\Z/2\Z}\Z/6\Z
 \]
 then places each finite subgroup of $G_1$, up to
 conjugacy, in $\langle J\rangle$ or $\langle PJ\rangle$.
 Since $J=W_4$, $PJ=W_6$, and $W_6^2=W_3$, the
 nontrivial subgroups give the four integral
 representatives above. Conjugating by $D_p$ gives
 the representatives in $G_2$. Finally, $-I$ is central
 and is fixed by this conjugation, so the order-two
 representative is the same in both copies.
\end{proof}

\subsection{Conjugacy of the integral and
\texorpdfstring{$p$}{p}-twisted representatives}

We now determine when the pairs of the same order are
conjugate. Write $H\sim_{\mathrm{conj}}K$ for conjugacy
of subgroups within $\mathrm{SL}_2(R_p)$.

\begin{prp}\label{prop:finite-representative-conjugacy}
 For any prime $p$,
 \begin{align*}
 \langle W_{4,p}\rangle
    &\sim_{\mathrm{conj}}\langle W_4\rangle
        &&\Longleftrightarrow\quad p=2\ \text{or}\ p\equiv1\pmod4,\\
 \langle W_{3,p}\rangle
    &\sim_{\mathrm{conj}}\langle W_3\rangle
        &&\Longleftrightarrow\quad
                p=3\ \text{or}\ p\equiv1\pmod3,\\
 \langle W_{6,p}\rangle
    &\sim_{\mathrm{conj}}\langle W_6\rangle
        &&\Longleftrightarrow\quad
                p=3\ \text{or}\ p\equiv1\pmod3.
 \end{align*}
 Whenever a condition holds, the displayed generators
 themselves are conjugate.
\end{prp}
\begin{proof}
 Suppose first that $C\in\mathrm{SL}_2(R_p)$ conjugates
 $\langle W_{4,p}\rangle$ to $\langle W_4\rangle$.
 It sends $W_{4,p}$ to $W_4$ or $W_4^{-1}$.
 Set $X=CD_p$, so $\det X=p$. The two possibilities
 give $XW_4=W_4X$ or $XW_4=W_4^{-1}X$.
 Solving for the entries of $X$ gives, respectively,
 \[
 \begin{aligned}
 X&=\begin{pmatrix}a&b\\-b&a\end{pmatrix},
       &\det X&=a^2+b^2,\\
 X&=\begin{pmatrix}a&b\\b&-a\end{pmatrix},
       &\det X&=-a^2-b^2.
 \end{aligned}
 \]
 The second determinant cannot equal $p>0$. Thus
 conjugacy requires $a^2+b^2=p$ with $a,b\in R_p$.
 Writing $a=A/p^N$ and $b=B/p^N$ gives
 \begin{equation}\label{eq:two-square-descent}
 A^2+B^2=p^{2N+1}\qquad(A,B\in\Z).
 \end{equation}
 If $p\equiv3\pmod4$, then $p\mid A^2+B^2$ forces
 $p\mid A,B$, since $-1$ is not a square modulo $p$.
 Repeated division by $p^2$ reduces
 \eqref{eq:two-square-descent} to a sum of two integer
 squares equal to $p$. The same divisibility argument
 then gives a contradiction.

 Conversely, for $p=2$ take $a=b=1$. For
 $p\equiv1\pmod4$, the Euclidean ring $\Z[i]$
 gives integers $a,b$ with $a^2+b^2=p$.
 To recall the argument, $X^2+1$ splits modulo $p$,
 so $\Z[i]/(p)$ has zero divisors. Hence $p$ is
 reducible in $\Z[i]$, and the norms of its two
 nonunit factors multiply to $p^2$. Both norms must
 therefore equal $p$. In either case, the commuting
 matrix $X$ above has determinant $p$, and
 $C=XD_p^{-1}$ belongs to $\mathrm{SL}_2(R_p)$ and
 sends $W_{4,p}$ to $W_4$.

 For order three, the same substitution $X=CD_p$
 gives $XW_3=W_3X$ or $XW_3=W_3^{-1}X$. The
 solutions are
 \[
 \begin{aligned}
 X&=\begin{pmatrix}a-b&b\\-b&a\end{pmatrix},
       &\det X&=a^2-ab+b^2,\\
 X&=\begin{pmatrix}-a&a+b\\b&a\end{pmatrix},
       &\det X&=-a^2-ab-b^2.
 \end{aligned}
 \]
 Again, the second determinant is nonpositive. The
 first must equal $p$, so clearing denominators gives
 \begin{equation}\label{eq:Eisenstein-descent}
 A^2-AB+B^2=p^{2N+1}\qquad(A,B\in\Z).
 \end{equation}
 Suppose $p\equiv2\pmod3$. If $p$ divides the left
 side and $B\not\equiv0\pmod p$, then $u=A/B$
 satisfies $u^2-u+1=0$ in $\mathbb F_p$. Hence
 $-u$ is a cube root of unity. It cannot equal $1$,
 since that would imply $3=0$ in $\mathbb F_p$.
 This contradicts $3\nmid p-1$. If $B\equiv0$,
 the congruence forces $A\equiv0$ as well.
 Thus $p\mid A,B$. Repeated division by $p^2$
 in \eqref{eq:Eisenstein-descent} leads to the same
 contradiction as above.

 For $p=3$, take $a=2$, $b=1$, so
 $a^2-ab+b^2=3$. For $p\equiv1\pmod3$, use
 the Eisenstein ring $\Z[\eta]$, where
 $\eta=e(1/3)$ and
 \[
 N(a+b\eta)=a^2-ab+b^2.
 \]
 This ring is Euclidean: rounding both coordinates
 of a quotient in the basis $1,\eta$ leaves errors
 of absolute value at most $1/2$, whose norm is at
 most $3/4<1$. Since $X^2+X+1$ splits modulo $p$,
 the quotient $\Z[\eta]/(p)$ has zero divisors.
 As in the Gaussian case, $p$ has a factor of norm
 $p$. Choose integers $a,b$ with $a^2-ab+b^2=p$.
 The commuting matrix $X$ above then has determinant
 $p$, and $C=XD_p^{-1}$ sends $W_{3,p}$ to $W_3$.

 Finally,
 \[
 W_6^2=W_3,\qquad W_6^3=-I,\qquad W_6=-W_3^2,
 \]
 with the same identities for the $p$-twisted matrices.
 Conjugacy of the order-six groups implies conjugacy of
 their unique order-three subgroups. Conversely,
 \[
 \langle W_6\rangle=\langle W_3,-I\rangle,\qquad
 \langle W_{6,p}\rangle=\langle W_{3,p},-I\rangle,
 \]
 so conjugacy of the order-three groups gives conjugacy
 of the order-six groups because $-I$ is central.
 The conjugator sending $W_{3,p}$ to $W_3$ also sends
 $W_{6,p}=-W_{3,p}^2$ to $W_6=-W_3^2$.
\end{proof}

\begin{cor}\label{cor:conjugate-cyclic-crossed-products}
 Let $\alpha\in\Omega_p^{\mathrm{even}}$.
 If $p=2$ or $p\equiv1\pmod4$, then
 \[
 \A_\alpha^\omega
   \rtimes_{\lambda_{W_{4,p}}^\alpha}(\Z/4\Z)
  \cong
 \A_\alpha^\omega
   \rtimes_{\lambda_{W_4}^\alpha}(\Z/4\Z).
 \]
 If $p=3$ or $p\equiv1\pmod3$, the analogous
 isomorphisms hold for the order-three and order-six
 pairs.
\end{cor}
\begin{proof}
 Proposition~\ref{prop:finite-representative-conjugacy}
 gives $C\in\mathrm{SL}_2(R_p)$ with
 $CW_{i,p}C^{-1}=W_i$ in each stated case.
 The action law implies
 \[
 \lambda_C^\alpha\lambda_{W_{i,p}}^\alpha
   =\lambda_{W_i}^\alpha\lambda_C^\alpha.
 \]
 Let $v$ and $u$ be the canonical group unitaries in
 the two crossed products. The assignments
 \[
 a\longmapsto\lambda_C^\alpha(a),\qquad
 v\longmapsto u
 \]
 preserve the covariance relation and the relation
 $v^i=1$. Since the acting group is finite, the full and reduced
 crossed products coincide. The universal property gives a
 homomorphism, and the same construction with
 $C^{-1}$ gives its inverse.
\end{proof}

For a standard parameter $\gamma\in\Omega_p$, choose an
even lift $\widetilde\gamma$ with
$h(\widetilde\gamma)=h(\gamma)$. Transport through
$\Phi_{\widetilde\gamma}$ gives the same isomorphisms for
$\A_\gamma$, with actions
$\delta_{W_{i,p}}^{\widetilde\gamma}$ and
$\delta_{W_i}^{\widetilde\gamma}$. Their isomorphism classes
are independent of the lift by
Corollary~\ref{cor:lift-independent-crossed-products}.
The congruence conditions above classify subgroup conjugacy;
they do not assert a converse for crossed-product
isomorphism or Morita equivalence.

\subsection{Implementers for the representatives}

Table~\ref{tab:finite-Weyl-implementers} gives matrix words and Weyl implementers
for the representatives. Here $P=L_1$,
$S_P=S_{L_1}$, and $D=\operatorname{diag}(p,p^{-1})$
is the determinant-one generator of
\eqref{eq:matrix-generators}. Each implementer is to be
normalized separately by
Lemma~\ref{lem:finite-implementer-normalization}.

\begin{table}[htbp]
 \centering
 \renewcommand{\arraystretch}{1.25}
 \begin{tabular}{lll}
 \hline
 Matrix & Word & Weyl implementer\\
 \hline
 $W_2$   & $J^2$       & $S_J^2$\\
 $W_3$   & $JP^{-1}$   & $S_JS_P^{-1}$\\
 $W_4$   & $J$         & $S_J$\\
 $W_6$   & $PJ$        & $S_PS_J$\\
 $W_{4,p}$ & $DJ$      & $S_DS_J$\\
 $W_{6,p}$ & $DP^pJ$   & $S_DS_P^pS_J$\\
 $W_{3,p}$ & $(DP^pJ)^2$
                          & $(S_DS_P^pS_J)^2$\\
 \hline
 \end{tabular}
 \caption{Words and unnormalized Weyl implementers.}
 \label{tab:finite-Weyl-implementers}
\end{table}

For example,
\[
 DP^pJ
  =\begin{pmatrix}p&0\\0&p^{-1}\end{pmatrix}
   \begin{pmatrix}1&0\\p&1\end{pmatrix}
   \begin{pmatrix}0&1\\-1&0\end{pmatrix}
  =\begin{pmatrix}0&p\\-p^{-1}&1\end{pmatrix}
  =W_{6,p},
\]
and its square is $W_{3,p}$. Under its stated hypotheses,
Theorem~\ref{finite Morita} applies to every matrix in the
table and to its conjugates.